\documentclass[10pt,final]{article}
\usepackage{a4}
\usepackage{amsmath}%
\usepackage{amstext}%
\usepackage{amssymb}%
\usepackage{amsthm}
\usepackage{comment}
\usepackage{listings}

\usepackage[final]{graphicx}
\usepackage{subcaption}

\usepackage[margin=0.5in]{geometry}

\usepackage{graphicx}
\usepackage{epstopdf}

\usepackage{xcolor}

\newtheorem{theorem}{Theorem}

\newtheorem{axiom}{Axiom}

\newtheorem{corollary}[theorem]{Corollary}

\newtheorem{definition}[axiom]{Definition}

\newtheorem{lemma}[theorem]{Lemma}

\newtheorem{proposition}[theorem]{Proposition}
\newtheorem{rem}[theorem]{Remark}
\newenvironment{remark}{\rem\rm}{\endrem}

\newcommand{\R}{\mathbb{R}}%
\newcommand{\N}{\mathbb{N}}%
\newcommand{\e}{\varepsilon}%
\newcommand{\ol}{\overline}%

\newcommand{\n}{{\nabla}}
\newcommand{\p}{{\partial}}
\newcommand{\ds}{\displaystyle}
\newcommand{\To}{\longrightarrow}
\def\a{\alpha}

\def\b{\beta}
\def\e{\epsilon}
\def\d{\delta}
\def\t{\theta}
\def\g{\gamma}

\def\l{\lambda}
\def\<{\langle}
\def\>{\rangle}
\DeclareMathOperator*\dom{dom}%
\DeclareMathOperator*\prox{prox}%
\DeclareMathOperator*\argmin{argmin}

\DeclareMathOperator*\crit{crit}
\DeclareMathOperator*\dist{dist}
\DeclareMathOperator*\sgn{sgn}

\DeclareMathOperator*\pr{pr}

\title{Forward-backward algorithms with different inertial terms for  structured non-convex  minimization problems}% of the sum of two non-convex functions}

\author{   Szil\'ard Csaba  L\'aszl\'o \thanks{Technical University of Cluj-Napoca, Department of Mathematics, Memorandumului 28, Cluj-Napoca, Romania, e-mail: szilard.laszlo@math.utcluj.ro., This work was supported by a grant of Ministry of Research and Innovation, CNCS -
UEFISCDI, project number PN-III-P1-1.1-TE-2016-0266}
}

\medskip

\begin{document}

\maketitle

\begin{abstract}
We investigate two inertial forward-backward  algorithms  in connection with the minimization of the sum of a  non-smooth and  possibly non-convex and a non-convex differentiable function.  The algorithms are formulated in the spirit of the famous FISTA method, however the setting is non-convex and we allow different inertial terms. Moreover, the inertial parameters in our algorithms can take negative values too.  We also treat the case when the non-smooth function is convex and we show that in this case a better step size can be allowed. We prove some abstract convergence results which applied to our numerical schemes allow us to show that the generated sequences converge to a critical point of the objective function, provided a regularization of the objective function satisfies the Kurdyka-{\L}ojasiewicz property.
 Further, we obtain a general result that applied to our numerical schemes ensures convergence rates for the generated sequences and for the objective function values formulated in terms of the KL exponent of a regularization of the objective function. Finally, we apply our results to image restoration.
 \end{abstract}

{\bf Key words:}
global optimization; inertial proximal-gradient algorithm; non-convex optimization; abstract convergence theorem; Kurdyka-\L{}ojasiewicz inequality; KL exponent; convergence rate

{\bf AMS subject classifications:}
90C26; 90C30; 65K10

\section{Introduction}

Let $f:\R^m\To\R\cup\{+\infty\}$ be a  proper and lower semicontinuous function and let $g:\R^m\To\R$ be a  smooth function with $L_g$ Lipschitz continuous gradient, that is, $\|\n g(x)-\n g(y)\|\le L_g\|x-y\|$ for all $x,y\in\R^m.$
Consider the optimization problem
\begin{equation}\label{propt}
\inf_{x\in\R^m}f(x)+g(x).
\end{equation}

We associate to this optimization problem the following forward-backward algorithm.
For the initial values $x_{-1}=x_0\in\R^m$ and for all $n\in\N$ consider
\begin{equation*}
\rm{(PADISNO)\,\,\,\,\,\,\,\,\,\,}\left\{\begin{array}{llll}
\ds y_n=x_n+\a_n(x_n-x_{n-1}),\\
\\
\ds z_n=x_n+\b_n(x_n-x_{n-1}),\\
\\
\ds x_{n+1}\in\argmin_{y\in\R^m}\left\{f(y)+\<\n g(z_n), y-y_n\>+\frac{1}{2s}\|y-y_n\|^2\right\},
\end{array}\right.
\end{equation*}
\noindent where PADISNO stands shortly for
"Proximal Algorithm with Different Inertial Steps for Nonconvex Optimization".
We assume that the sequences $(\a_n)_{n\in\N}$ and $(\b_n)_{n\in\N}$ and the step size $s$ in the definition of PADISNO satisfy the following conditions:
$$\boxed{\a_n\To\a\in\left(-\frac{1}{2},\frac{1}{2}\right),\,\b_n\To\b\in\R,\,n\To+\infty\mbox{ and }0<s<\frac{1-2|\a|}{L_g(2|\b|+1)}.}$$

Further, we assume  that the function $f$ is bounded from  below  in order to ensure that the  argmin set in the definition of $x_{n+1}$ is nonempty.
Indeed, note that
$$\argmin_{y\in\R^m}\left\{f(y)+\<\n g(z_n), y-y_n\>+\frac{1}{2s}\|y-y_n\|^2\right\}=\argmin_{y\in\R^m}\left\{f(y)+\frac{1}{2s}\|y-(y_n-s\n g(z_n))\|^2\right\},$$
and obviously if $f$ is bounded from below then the function $\psi(y)= f(y)+\frac{1}{2s}\|y-(y_n-s\n g(z_n))\|^2$ is coercive, (i.e. $\lim_{\|y\|\To+\infty}\psi(y)=+\infty$), hence $\argmin_{y\in\R^m}\psi(y)\neq\emptyset$.

 We underline that in case $g$ is a concave function we may allow in PADISNO the step size $0<s<\frac{1-2|\a|}{2Lg|\b|},$ (with the convention that the right hand side of the previous inequality is $+\infty$ for $\b=0$), hence for an appropriate choice of the inertial parameters $(\a_n)$ and $(\b_n)$ the step size can be arbitrary large.

Observe that we allow different extrapolation terms in PADISNO, moreover the inertial coefficients $\a_n,\,n\in \N$ and $\b_n,\,n\in\N$ can take negative values too.
Let us discuss the relation of our scheme with other algorithms from the literature.

First of all, note that  PADISNO can  equivalently be written as
\begin{equation}\label{compalg}
x_{n+1}\in\argmin_{y\in\R^m}\left\{f(y)+\frac{1}{2s}\|y-(x_n+\a_n(x_n-x_{n-1})-s\n g(x_n+\b_n(x_n-x_{n-1})))\|^2\right\}.
\end{equation}

If we take $\b_n\equiv 0$ then \eqref{compalg} becomes a particular case of the algorithm studied in \cite{BCL}. Further, if we assume $\a_n\equiv0$ and $\b_n\equiv 0$ then
\eqref{compalg} leads to the algorithm
\begin{equation}\label{compalg1}
x_{n+1}\in\argmin_{y\in\R^m}\left\{f(y)+\frac{1}{2s}\|y-(x_n-s\n g(x_n))\|^2\right\},
\end{equation}
that was investigated in \cite{b-sab-teb}, see also \cite{CP} and \cite{FGP}, (see also \cite{HLMQY}, where \eqref{compalg1} was investigated in connection to a particular instance of \eqref{propt}).
If we assume additionally that $g\equiv 0$ then we obtain the algorithm studied in \cite{attouch-bolte2009}.

Consider now the case $f\equiv 0.$ Then \eqref{compalg} becomes
\begin{equation}\label{compalg0}
x_{n+1}=x_n+\a_n(x_n-x_{n-1})-s\n g(x_n+\b_n(x_n-x_{n-1})),
\end{equation}
 which is the algorithm obtained in \cite{ALP} from the explicit discretization of a perturbed heavy ball system. Further, if $\a_n=\b_n=\frac{\b n}{n+\a}$ for all $n\in\N$, where $\b\in(0,1)$ and $\a>0$, then \eqref{compalg0} becomes the algorithm studied in \cite{L}. Moreover, if $\a_n=\b_n$ for all $n\in\N$ and $\lim_{n\To+\infty}\a_n=\a\in\left(\frac{-10+\sqrt{68}}{8},0\right)$ and $s<\frac{4\a^2+10\a+2}{L_g(2\a+1)^2}$ then \eqref{compalg0} leads to the algorithm studied in \cite{ALV}.

 We refer to \cite{BauComb}, \cite{bete09} , \cite{MO} for the full convex case, that is, the functions $f$ and $g$ are convex,  where different instances of PADISNO have been investigated. For other inertial optimization algorithms studied in the literature we refer to \cite{AAD,ALV,alvarez-attouch2001,APR,bete09,BCH,ch-do2015,CG,GFJ,GRV,L,LRP,PL,nesterov83,poly,SYLHGJ,ZK}.
\vskip0.3cm

Let us consider now a variant of PADISNO where we assume that the function $f$ is convex, see also \cite{bc-forder-kl} for a continuous counterpart. In this case PADISNO can be written as follows.
For $x_{-1}=x_0\in\R^m$ and $n\in\N$ consider
\begin{equation*}
\rm{(c-PADISNO)\,\,\,\,\,\,\,\,\,\,}\left\{\begin{array}{llll}
\ds y_n=x_n+\a_n(x_n-x_{n-1}),\\
\\
\ds z_n=x_n+\b_n(x_n-x_{n-1}),\\
\\
\ds x_{n+1}=\prox\nolimits_{s f}(y_n-s\n g(z_n)).
\end{array}\right.
\end{equation*}

Here $\prox\nolimits_{s  f} : {\R^m} \rightarrow {\R^m}, \quad \prox\nolimits_{s f}(x)=\argmin_{y\in {\R^m}}\left \{f(y)+\frac{1}{2s}\|y-x\|^2\right\},$ denotes the proximal point operator of the convex function $s  f$. %We recall the well known identity between the proximal point operator of the function $f$ and the resolvent operator of its subdifferential $\p f$ that is, the equality
%$\prox\nolimits_{f}(x)=(I+\p f)^{-1}(x)$ holds for all $x\in \R^m$.%, where $I:\R^m\To\R^m$ denotes the identity operator and $\p f$ denotes the subdifferential of the convex function $f$.

The assumption that the function $f$ is convex in c-PADISNO allows us to consider  some more general forms for the sequences $(\a_n)_{n\in\N},(\b_n)_{n\in\N}$ and also for the step size $s.$ More precisely we may allow the following conditions:
$$\boxed{\a_n\To\a\in(-1,1),\,\b_n\To\b\in\R,\,n\To+\infty\mbox{ and }0<s<\frac{2(1-|\a|)}{L_g(2|\b|+1)}.}$$

 Therefore, we emphasize that despite its similar formulation, c-PADISNO is not entirely a particular case of PADISNO, the assumption that $f$ is convex leads to a much better step size in the latter. Further, in c-PADISNO we do not need to assume that the function $f$ is bounded from below, since the function $y\To f(y)+\frac{1}{2s}\|y-x\|^2$ is strongly convex for all $x\in\R^m$ and therefore the proximal operator (and consequently  $x_{n+1}$) is defined everywhere.

  Let us notice that if we assume that $g$ is a concave function we are in the DC programming framework, since our objective function in the optimization problem  \eqref{propt} is the difference of two convex functions, $f-(-g)$. In this case we may allow in c-PADISNO the step size $0<s<\frac{1-|\a|}{Lg|\b|},$ hence also here, for an appropriate choice of the inertial parameters $(\a_n)$ and $(\b_n)$ the step size can be arbitrary large. Consequently, c-PADISNO can be successfully applied to DC programming problems where the objective function is the diffrence of a non-smooth convex and a smooth convex function, (for some similar approaches see \cite{NOSS}).

  A first variant of c-PADISNO, with the inertial parameters $(\a_n)\subseteq[0,a),\,a<1,\,(\b_n)\subseteq[0,1]$ and variable step size $\l_n$ satisfying $\e_1\le \l_n<\frac{2}{L_g}-\e_2,\,\e_1,\e_2>0$ has been introduced in \cite{LFP}, and also studied in the full convex case, (i.e. the function $g$ is also convex), in \cite{JM}, under some more general convergence conditions. The results from \cite{JM} were extended to the case when $g$ is non-convex in \cite{WL}, (see also \cite{WLLL}). The novelty of c-PADISNO consists in allowing the inertial parameters to take negative values too, more precisely $(\a_n)\subseteq(-1,1)$  and $(\b_n)\subseteq \R$, thus the sequences generated by c-PADISNO will have a superior convergence behavior as some numerical experiments show, (see Section 1.1)

Let us mention that for us the  idea of using different extrapolation terms comes from  \cite{ALP}, where the authors considered a perturbed heavy ball system with vanishing damping, (see, also \cite{su-boyd-candes})
\begin{equation}\label{dysyALP}
\left\{
\begin{array}{ll}
\ddot{x}(t)+\frac{\a}{t}\dot{x}(t)+\nabla g\left(x(t)+\left(\g+\frac{\b}{t}\right)\dot{x}(t)\right)=0\\
x(t_0)=u_0,\,\dot{x}(t_0)=v_0,\, u_0,v_0\in \R^m,\,t_0>0\mbox{ and }\a>0,\,\b\in\R,\,\g\ge 0,
\end{array}
\right.
\end{equation}
in connection to the convex optimization problem $\inf_{x\in\R^m}g(x). $
%where $g:\R^m\To \R$ is a convex Fr\'{e}chet differentiable, function with $L_g$-Lipschitz continuous gradient.

According to \cite{ALP}, explicit Euler discretization of \eqref{dysyALP} leads to the algorithm
\begin{equation}\label{alggentemp}
\left\{
\begin{array}{lll}
y_n=x_n+\left(1-\frac{\a}{n}\right)(x_n-x_{n-1})\\
z_n=x_n+\left({\g}+\frac{\b}{n}\right)(x_n-x_{n-1})\\
x_{n+1}=y_n-s\n g\left(z_n\right),
\end{array}
\right.
\end{equation}
where $x_0,x_{-1}\in\R^m$ and the iterations are considered for all $n\in\N$.
Notice that Algorithm \eqref{alggentemp} seems to be a particular case of c-PADISNO for $f\equiv 0.$

One can easily observe that for $\b_n\equiv 0$ c-PADISNO becomes a version of iPiano studied in \cite{OCBP,Ochs}.
We underline that c-PADISNO has a  similar formulation as the  FISTA algorithm,  see \cite{bete09,ch-do2015}, but we allow for different inertial terms in order to get a better control on the step size $s$. Consequently, the convergence of the generated sequences to a critical point of the objective function $f+g$  opens the gate for the study of FISTA type algorithms in a non-convex setting. However, our analysis do not cover the case $\a_n\To 1$, essential for FISTA type algorithms, nevertheless, the numerical experiments from the next section suggest that for c-PADISNO this requirement is not essential.

%\subsection{The main convergence results}

In what follows the main results of the paper are stated. The first one assures the convergence of the sequences generated by the algorithms PADISNO and c-PADISNO. The second result provides convergence rates.

\begin{theorem}\label{convergence} In the settings of problem \eqref{propt}, for some starting points $x_{-1}=x_0\in\R^m,$ consider the sequence $(x_n)_{n\in\N}$ generated by PADISNO or c-PADISNO. Assume that $f+g$ is bounded from below and consider the function
$$H:\R^m\times\R^m\To\R\cup\{+\infty\},\,H(x,y)=(f+g)(x)+\frac12\|y-x\|^2.$$ Let $x^*$ be a cluster point of the sequence $(x_n)_{n\in\N}$ and assume that $H$ satisfies the Kurdyka-{\L}ojasiewicz property at the point $z^*=(x^*,x^*).$
%Then, the sequence $(x_n)_{n\in\N}$ converges to $x^*$, $(z_n)_{n\in\N}$ converges to $z^*$ and $z^*\in\crit(F).$

 Then, the sequence $(x_n)_{n\in\N}$ converges to $x^*$ and $x^*$ is a critical point of the objective function $f+g.$
\end{theorem}
\begin{theorem}\label{convergencerates} In the settings of problem \eqref{propt}, for some starting points $x_{-1}=x_0\in\R^m,$ consider the sequence $(x_n)_{n\in\N}$ generated by PADISNO or c-PADISNO. Assume that $f+g$ is bounded from below and consider the function
$$H:\R^m\times\R^m\To\R\cup\{+\infty\},\,H(x,y)=(f+g)(x)+\frac12\|y-x\|^2.$$ Let $x^*$ be a cluster point of the sequence $(x_n)_{n\in\N}$ and assume that $H$ has the Kurdyka-{\L}ojasiewicz property at the point $z^*=(x^*,x^*),$ with the KL exponent $\t\in[0,1).$
%Then, the sequence $(x_n)_{n\in\N}$ converges to $x^*$, $(z_n)_{n\in\N}$ converges to $z^*$ and $z^*\in\crit(F).$

 Then, the sequence $(x_n)_{n\in\N}$ converges to $x^*$ and $x^*$ is a critical point of the objective function $f+g.$

Further, the following statements hold.
\begin{itemize}
\item[$(\emph{a})$] If $\t=0$ then $((f+g)(x_n))_{n\in\N},(x_n)_{n\in\N},(y_n)_{n\in\N}$ and $(z_n)_{n\in\N}$ converge in a finite number of steps;
\item[$(\emph{b})$] If $\t\in\left(0,\frac12\right]$ then there exist $a_1>0,$ $Q\in[0,1)$  and $\ol N\in\N$ such that
$(f+g)(x_n)-(f+g)(x^*)\le a_1 Q^n$, $\|x_n-x^*\|\le a_1 Q^{\frac{n}{2}}$, $\|y_n-x^*\|\le a_1 Q^{\frac{n}{2}}$ and $\|z_n-x^*\|\le a_1 Q^{\frac{n}{2}}$ for  every $n\ge \ol N$;
\item[$(\emph{c})$] If $\t\in\left(\frac12,1\right)$ then there exist  $a_2>0$ and $\ol N\in\N$ such that
$(f+g)(x_n)-(f+g)(x^*)\le a_2 {n^{-\frac{1}{2\t-1}}},$ $\|x_n-x^*\|\le a_2 {n^{\frac{\t-1}{2\t-1}}}$, $\|y_n-x^*\|\le a_2 {n^{\frac{\t-1}{2\t-1}}}$  and $\|z_n-x^*\|\le a_2 {n^{\frac{\t-1}{2\t-1}}}\mbox{ for all }n\ge\ol N$.
\end{itemize}
\end{theorem}

\subsection{Motivating numerical experiments}

Before we start with the convergence analysis we present some numerical experiments in order to emphasize the usefulness of considering different inertial steps and the usefulness of allowing negative inertial parameters in the Algorithm c-PADISNO. Our numerical experiments reveal that indeed c-PADISNO has a remarkable behavior.

Consider the convex non-smooth function $f:\R^2\To\R,\, f(x,y)=\sqrt{(x^2+y^2)^3}.$ Then, according to  \cite{beck}, the proximal operator of $\l f,\,\l>0$ is given by $\prox\nolimits_{\l f}:\R^2\To\R^2,$
$$\prox\nolimits_{\l f}(x,y)=\left(\frac{2}{1+\sqrt{1+12 \l\sqrt{x^2+y^2}}}x,\frac{2}{1+\sqrt{1+12 \l\sqrt{x^2+y^2}}}y\right).$$
Consider further the function $g:\R^2\To\R,\, g(x,y)= (x^2 -y)^2 +x^2.$ Then, $g$ is not convex, (actually $f+g$ is non-convex), further $f+g$ has a global minimum at $x^*=(0,0)$, hence  $\min(f+g)=(f+g)(x^*)=0.$

Note that $\n g(x,y)=(4x^3-4xy+2x,2y-2x^2)$ is not globally Lipschitz, however since we will run c-PADISNO with the starting points $x_{-1}=x_0\in D=\{(x,y)\in\R^2:-1\le x,y\le 1\}$, for our goal is enough to compute the Lipschitz constant of $\n g$ on the set $D.$
Indeed, we show (see Figure 1), that for  starting points from $D$ the sequence $(x_n)_{n\in\N}$ generated by c-PADISNO remains in the interior of  $D.$

We have
$\n^2g(x,y)=\left(
               \begin{array}{cc}
                 12x^2-4y+2 & -4x \\
                 -4x & 2 \\
               \end{array}
             \right)$ and
the eigenvalues of $\n^2 g(x,y)$ are
$\l_1(x,y)=8x^2-4y+2$ and $\l_2(x,y)=4x^2+2.$
Hence, since $\n^2 g(x,y)$ is symmetric for all $(x,y)\in D$, we get that
$$\|\n^2 g(x,y)\|=8x^2-4y+2,\mbox{ if }x^2\ge y$$ and
$$\|\n^2 g(x,y)\|=4x^2+2,\mbox{ if }x^2< y.$$
Consequently, we have $\sup_{(x,y)\in D}\|\n^2 g(x,y)\|=14$ which shows that we may take the Lipschitz constant of $\n g$ on $D$ as $L_g=14$.

Observe that $f+g$ is coercive, further note that the functions $f$ and $g$ are semi-algebraic functions, hence $f+g$ is a semi-algebraic function. This fact ensures that also the function $H$ defined in the hypotheses of Theorem \ref{convergence} is semi-algebraic, therefore is a KL function. Consequently, according to Corollary \ref{fornumex}, (see Section 4), $(x_n)$ the sequence generated by c-PADISNO converges to $x^*$ as $n\To +\infty.$

\begin{comment}By using the inequalities $\|X+Y\|^2\le 2\|X\|^2+2\|Y\|^2$ for all $X,Y\in \R^2$ and $|XY-UV|^2=|XY-XV+XV-UV|^2\le 2X^2|Y-V|^2+2V^2|X-U|^2$ for all $X,Y,U,V\in\R$, we have
\begin{align*}
\nonumber \|\n g(x,y)-\n g(u,v)\|&=\sqrt{|4(x^3-u^3)-4(xy-uv)+2(x-u)|^2+|2(y-v)-2(x^2-u^2)|^2}\\
\nonumber&\le\sqrt{32|x^3-u^3|^2+64|xy-uv|^2+16|x-u|^2+8|y-v|^2+8|x^2-u^2|}\\
\nonumber&\le\sqrt{(32(x^2+xu+u^2)^2+128v^2+16+8(x+u)^2)|x-u|^2+(128 x^2+8)|y-v|^2}\\
\nonumber&\le \sqrt{464|x-u|^2+136|y-v|^2}\le 22\|(x,y)-(u,v)\|,
\end{align*}
 for all $(x,y),(u,v)\in D.$

Consequently, on $D$ we may take the Lipschitz constant $L_g=22$.
\end{comment}

In the following experiments, we consider the numerical scheme c-PADISNO with the inertial sequences
$$\a_n=\frac{\a n}{n+3.1},\,\b_n=\frac{\b n}{n+3.1},\,n\in\N,\,\a\in(-1,1),\,\b\in\R.$$
 Then, obviously $\a_n\To\a,\,\b_n\To\b,\,n\To+\infty$
 Further, we consider several instances for the constants $\a$ and $\b$  and we take the step size $$s=\frac{7}{50}\cdot\frac{1-|\a|}{2|\b|+1}<\frac{2(1-|\a|)}{L_g(2|\b|+1)}.$$

{\bf 1.} For showing that for starting points from $D$, the sequence $x_n=(x_n^1,x_n^2)$ generated by c-PADISNO remains in $D$, we run c-PADISNO until the error $\|x_n-x^*\|$  becomes less than the value $10^{-300}.$ The numerical experiment also reveals that indeed $(x_n)$, the sequence generated by c-PADISNO, converges to $x^*$. We consider the starting points $x_{-1}=x_0=\left(\frac12,-\frac12\right)$ from $D$ and for the parameters  $\a$ and $\b$  we consider the following instances:
$$\a\in\left\{-\frac{9}{10},-\frac12,0,\frac12,\frac35,\frac{9}{10}\right\},$$
$$\b\in\left\{-2,-\frac{3}{2},-1,-\frac{3}{4},-\frac12,-\frac{1}{1},0,\frac{1}{4}\frac12,\frac{3}{4},1\frac{3}{2},2\right\}.$$
\begin{comment}
$$(\a,\b,s)\in\left\{\left(\pm \frac12,\pm \frac12,\frac{7}{200}\right),\left(\pm \frac12,0,\frac{7}{100}\right),\left(0,0,\frac{7}{50}\right),\left(0,\pm \frac12,\frac{7}{100}\right)\right\},$$
see Figure 1 (a) for the starting points $x_{-1}=x_0=(0.5,-0.5)$ and Figure 1 (c) for the starting points $x_{-1}=x_0=(-0.9,-0.9)$, and
$$(\a,\b,s)\in\left\{\left(\pm \frac{3}{5},\pm \frac{1}{10},\frac{7}{150}\right), \left(\pm \frac{9}{10},\pm \frac{1}{4},\frac{7}{750}\right),\left( \frac{9}{10},0,\frac{7}{500}\right),\left(1,1,\frac{1}{14}\right)\right\},$$
see Figure 1 (b) for the starting points $x_{-1}=x_0=(0.5,-0.5)$ and Figure 1 (d) for the starting points $x_{-1}=x_0=(-0.9,-0.9)$.
\end{comment}

Note that the cases $\b=0$  correspond to the i-PIANO method \cite{OCBP}. Further, for comparison purposes, we implemented the case $\a=\b=1,\,s=\frac{1}{L_g}=\frac{1}{14},$ which corresponds to FISTA method \cite{bete09}.

The terms of the sequences $x_n=(x_n^1,x_n^2),\,n\in\N$ generated by c-PADISNO for these parameters and starting points are  depicted at Figure 1.

\begin{figure}[hbt!]\label{fig1}
\begin{subfigure}{.33\textwidth}
  \centering
  \includegraphics[width=.99\linewidth]{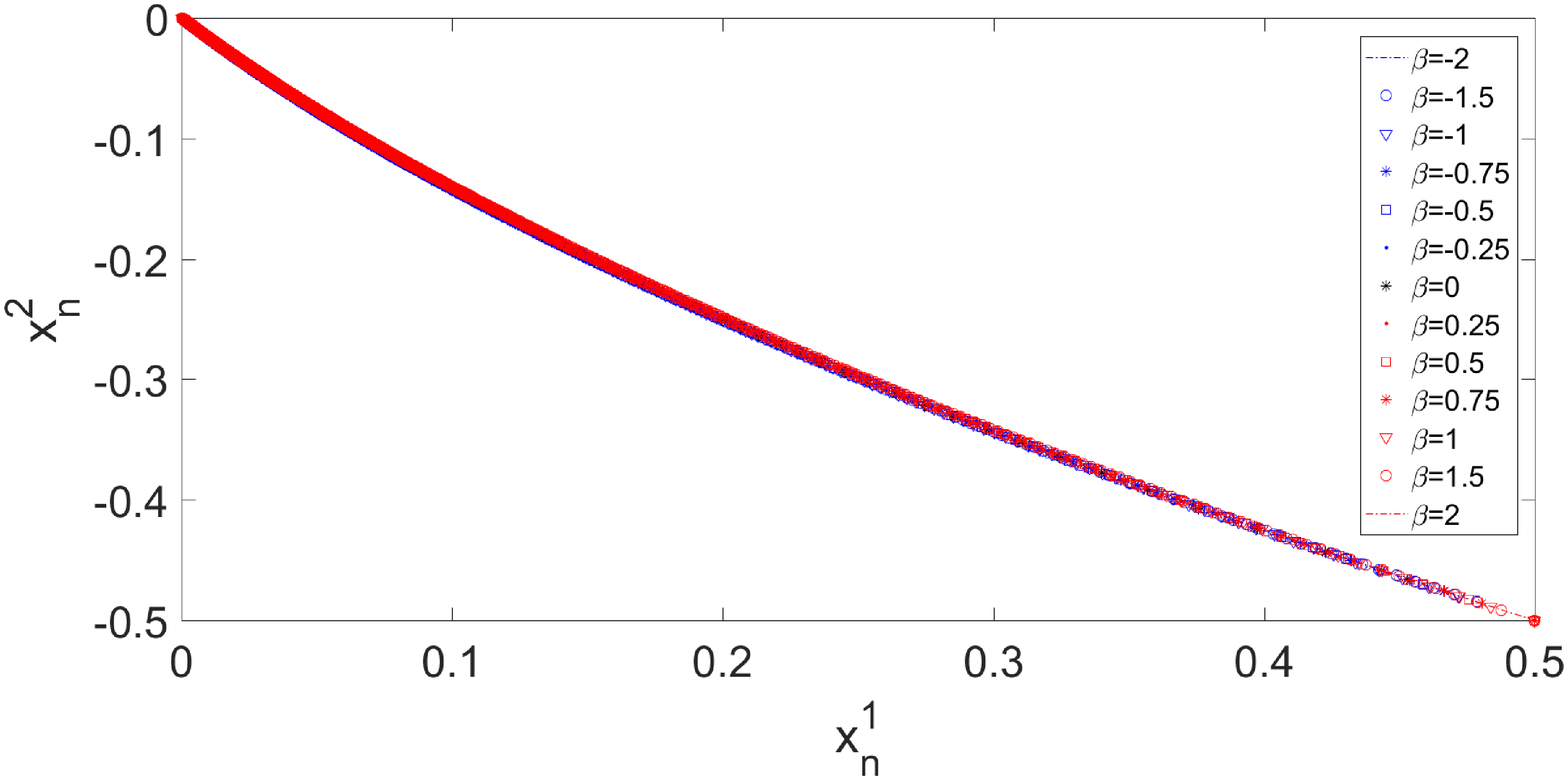}
  \caption{$\a=-0.9$}
\end{subfigure}
\begin{subfigure}{.33\textwidth}
  \centering
  \includegraphics[width=.99\linewidth]{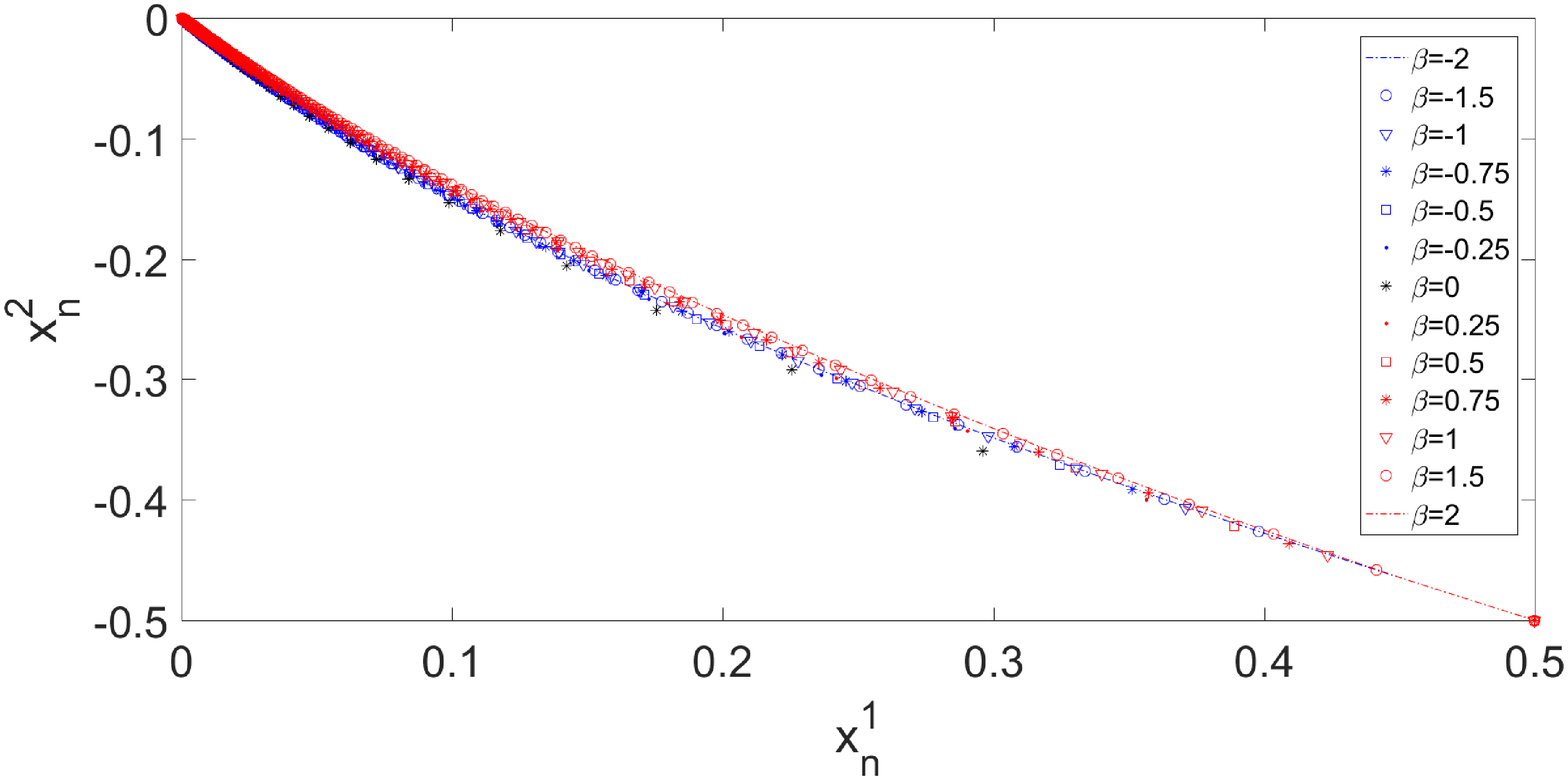}
  \caption{$\a=-0.5$}
\end{subfigure}
\begin{subfigure}{.33\textwidth}
  \centering
  \includegraphics[width=.99\linewidth]{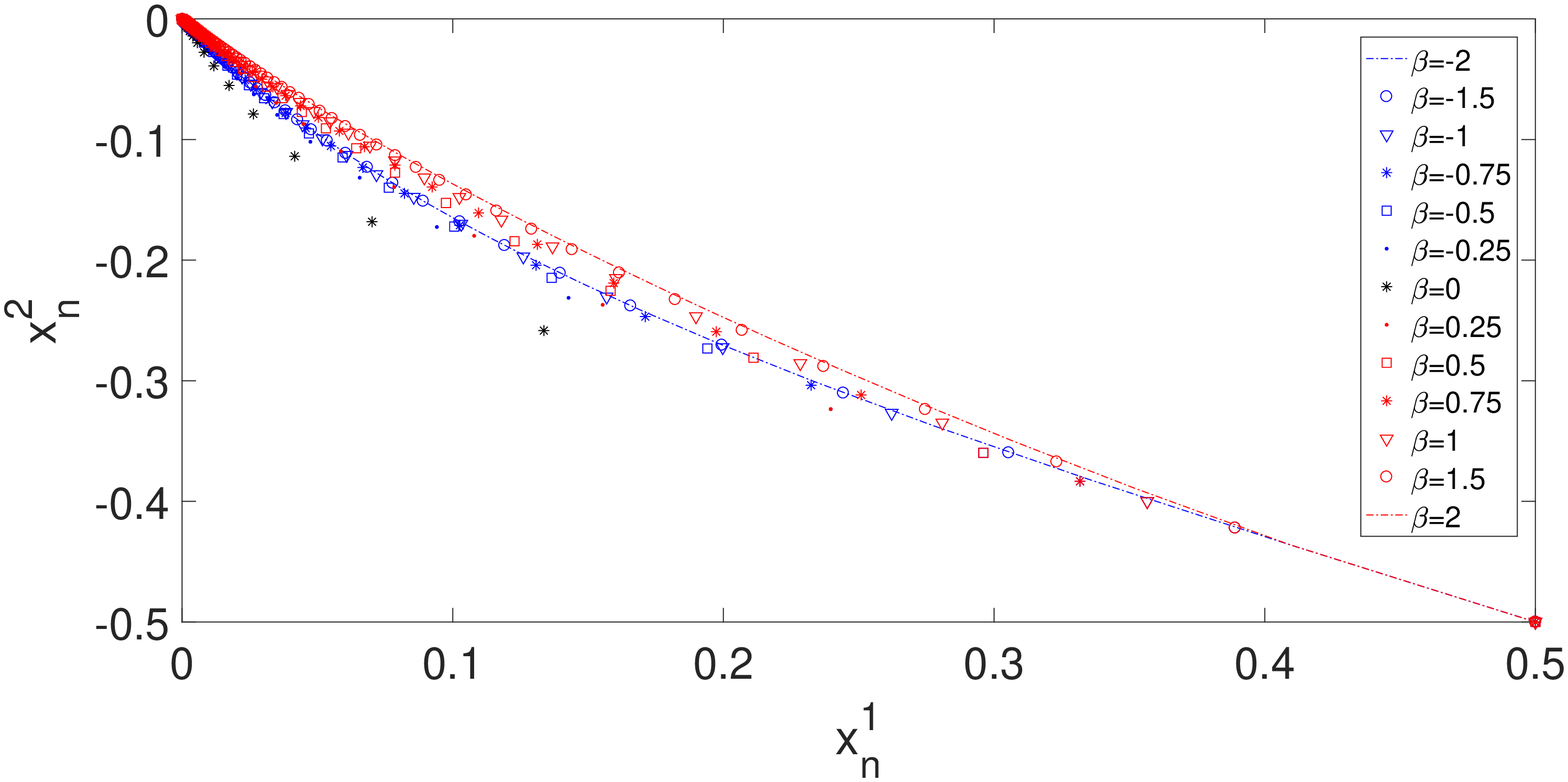}
  \caption{$\a=0$}
\end{subfigure}

\begin{subfigure}{.33\textwidth}
  \centering
  \includegraphics[width=.99\linewidth]{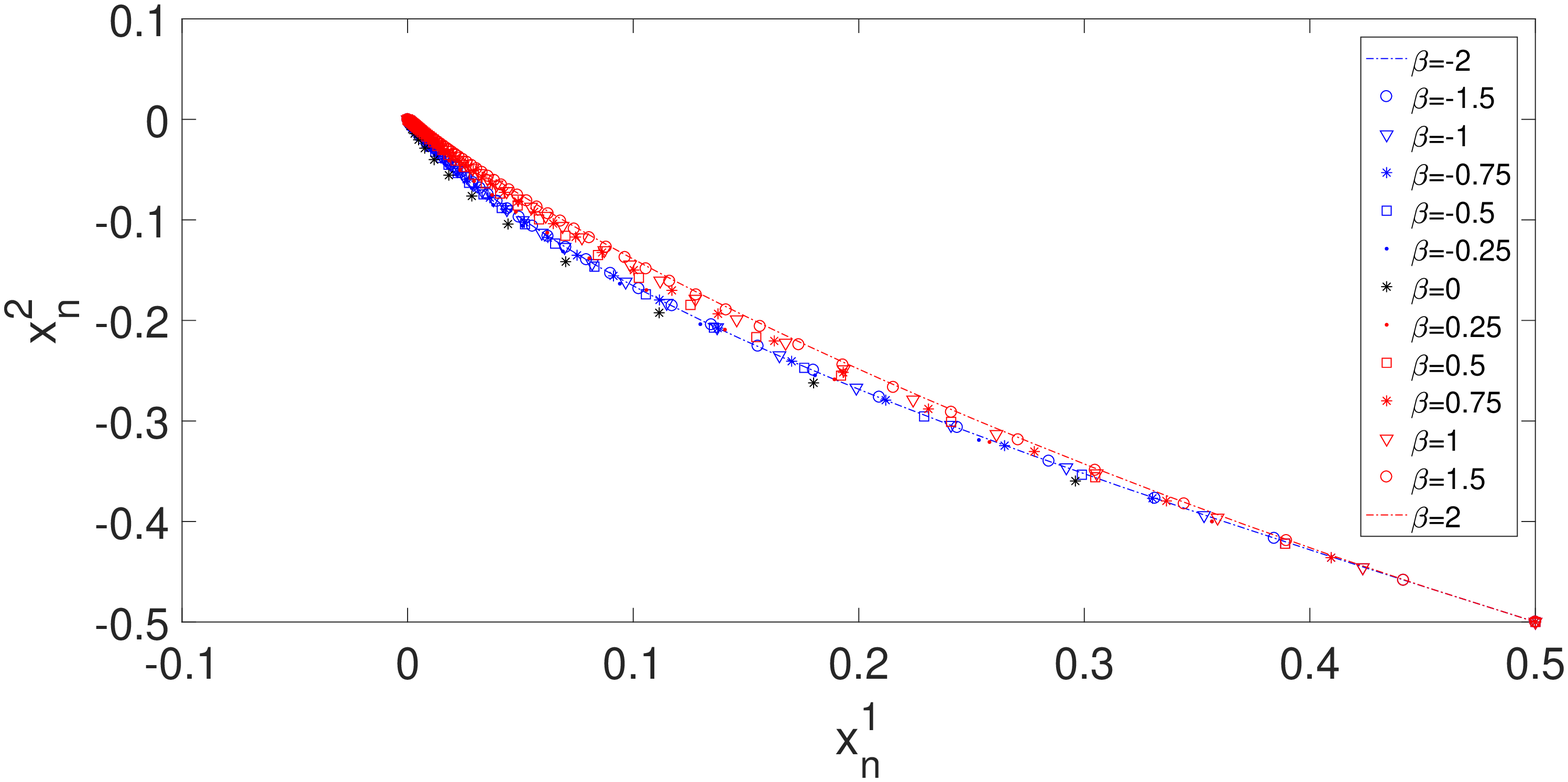}
  \caption{$\a=0.5$}
\end{subfigure}
\begin{subfigure}{.33\textwidth}
  \centering
  \includegraphics[width=.99\linewidth]{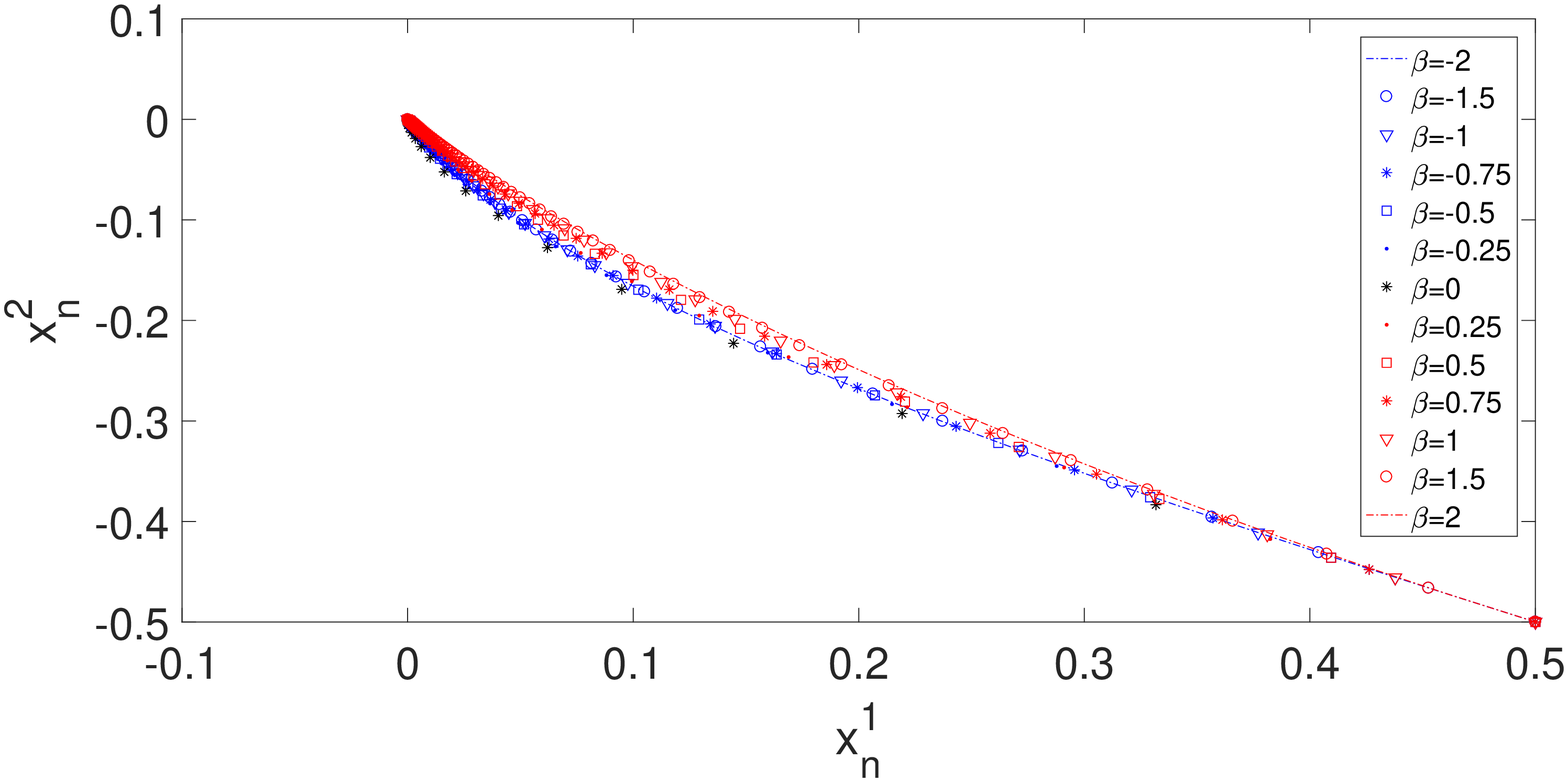}
  \caption{$\a=0.6$}
\end{subfigure}
\begin{subfigure}{.33\textwidth}
  \centering
  \includegraphics[width=.99\linewidth]{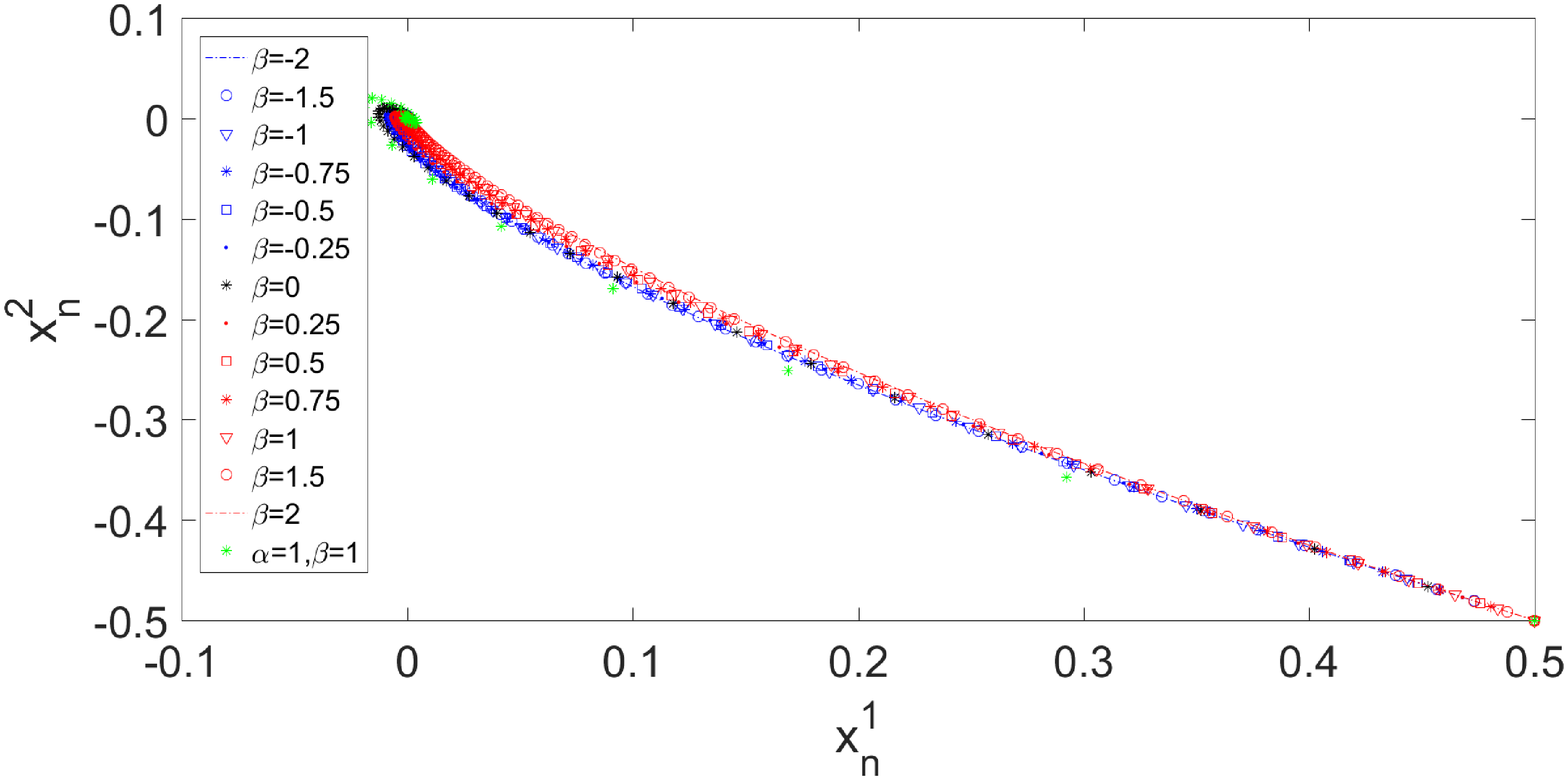}
  \caption{$\a=0.9$ and $\a=\b=1$}
\end{subfigure}
 \caption{Convergence analysis of the sequences generated by c-PADISNO with different inertial parameters.}
\end{figure}

{\bf 2.} In the next experiments, for the starting points $x_{-1}=x_0=\left(\frac12,-\frac12\right)$ we run c-PADISNO until the error $(f+g)(x_n)-\min(f+g)$ becomes less than the value $10^{-300}$  and the error $\|x_n-x^*\|$ becomes less than the value $10^{-300},$ respectively.  %where for the constants $\a$ and $\b$ we consider several instances and we take the step size $s=0.09\frac{1-|\a|}{2|\b|+1}<\frac{2(1-|\a|)}{L_g(2|\b|+1)}$ (see the legend of Figure 1).
Based on our experiments we conclude the following.
\vskip0.3cm
{\bf A. For a fixed inertial parameter $\a_n$, by choosing the inertial parameter $\b_n<0,$  c-PADISNO may have a better convergence behaviour than in the case $\b_n>0.$}

Indeed, the next experiment, depicted at Figure 2, shows the usefulness of allowing negative values for the inertial parameters. Note that taking negative inertial parameters can be thought as we make a backward inertial step. Although there is no geometrical interpretation for this, the experiment reveals that considering negative inertial parameters in c-PADISNO we may obtain a better behaviour of the algorithm than in the case when we use positive inertial parameters.

\begin{figure}[hbt!]
\begin{subfigure}{.5\textwidth}
  \centering
  \includegraphics[width=.99\linewidth]{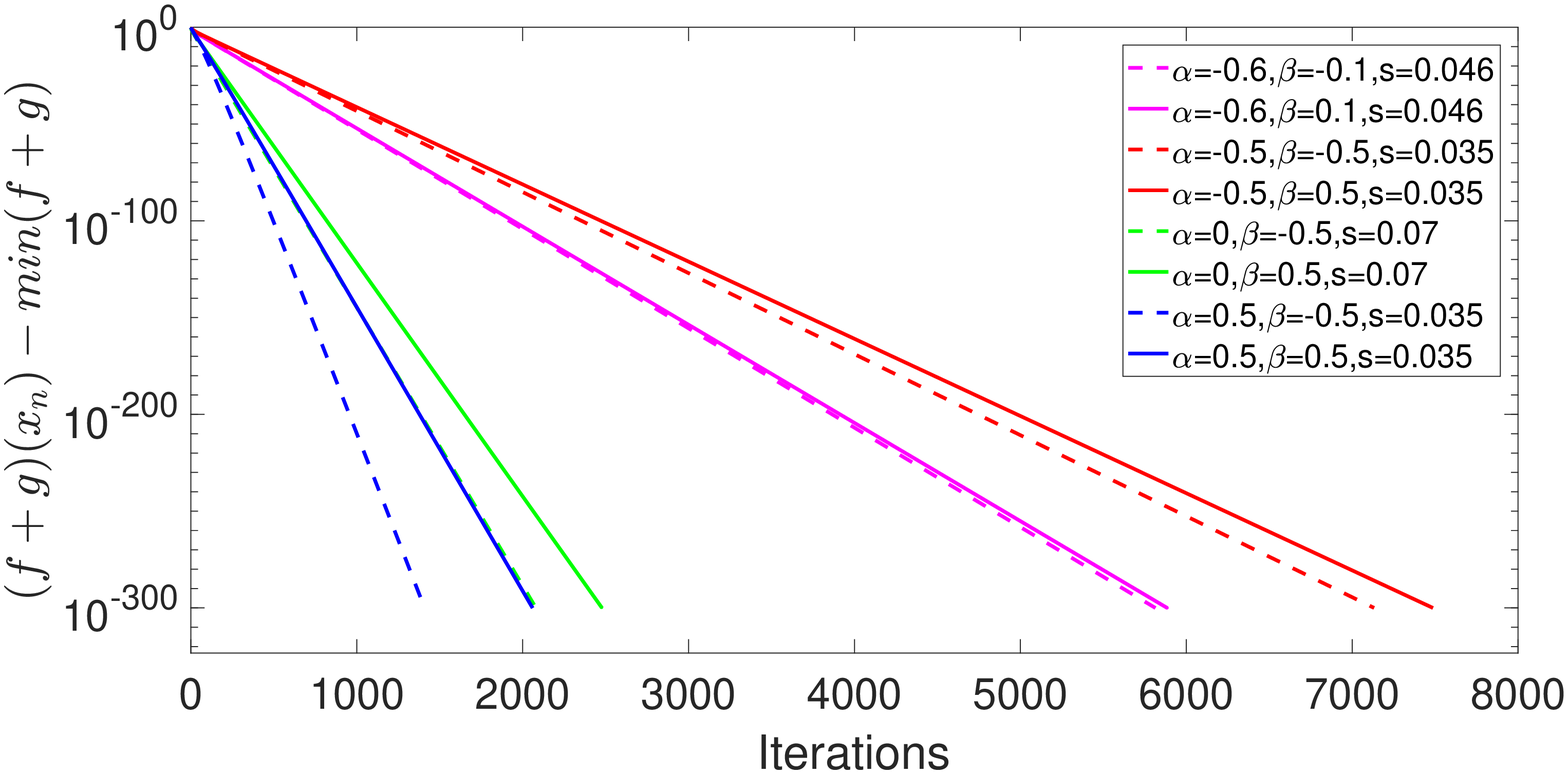}
  \caption{$x_{-1}=x_0=(0.5,-0.5).$}
  \label{fig2:sfig21}
\end{subfigure}
\begin{subfigure}{.5\textwidth}
  \centering
  \includegraphics[width=.99\linewidth]{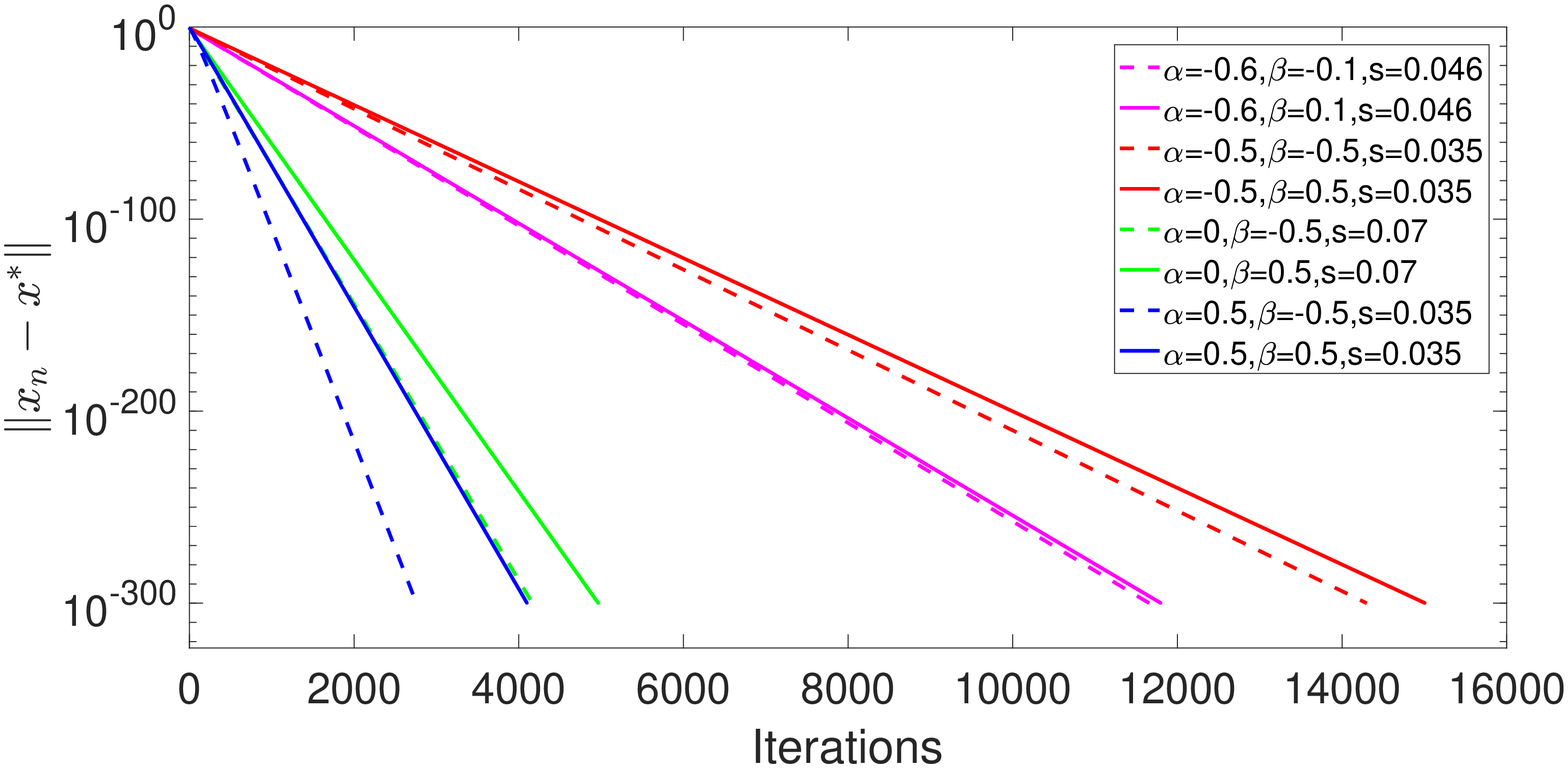}
  \caption{ $x_{-1}=x_0=(0.5,-0.5).$}
  \label{fig2:sfig22}
\end{subfigure}
 \caption{Allowing negative $\b_n$ in c-PADISNO may result in better performance.}
\end{figure}

{\bf B.  The convergence behaviour of the sequences generated by c-PADISNO in case of the inertial parameters $\a_n<0,\,\b_n<0$ may outperform the case $\a_n>0,\,\b_n>0$, even if in the latter case we  allow a better step size.}

In the next experiment, depicted at Figure 3,  we show that even if we choose negative inertial parameters $\a_n$ and $\b_n$ in c-PADISNO, we may obtain a better behaviour of the algorithm than in the case when we use positive inertial parameters. Surprisingly we obtain a better performance even if the step size in the case of negative $\a_n$ and $\b_n$ is worse than the step size of the case of positive $\a_n$ and $\b_n.$

\begin{figure}[hbt!]
\begin{subfigure}{.5\textwidth}
  \centering
  \includegraphics[width=.99\linewidth]{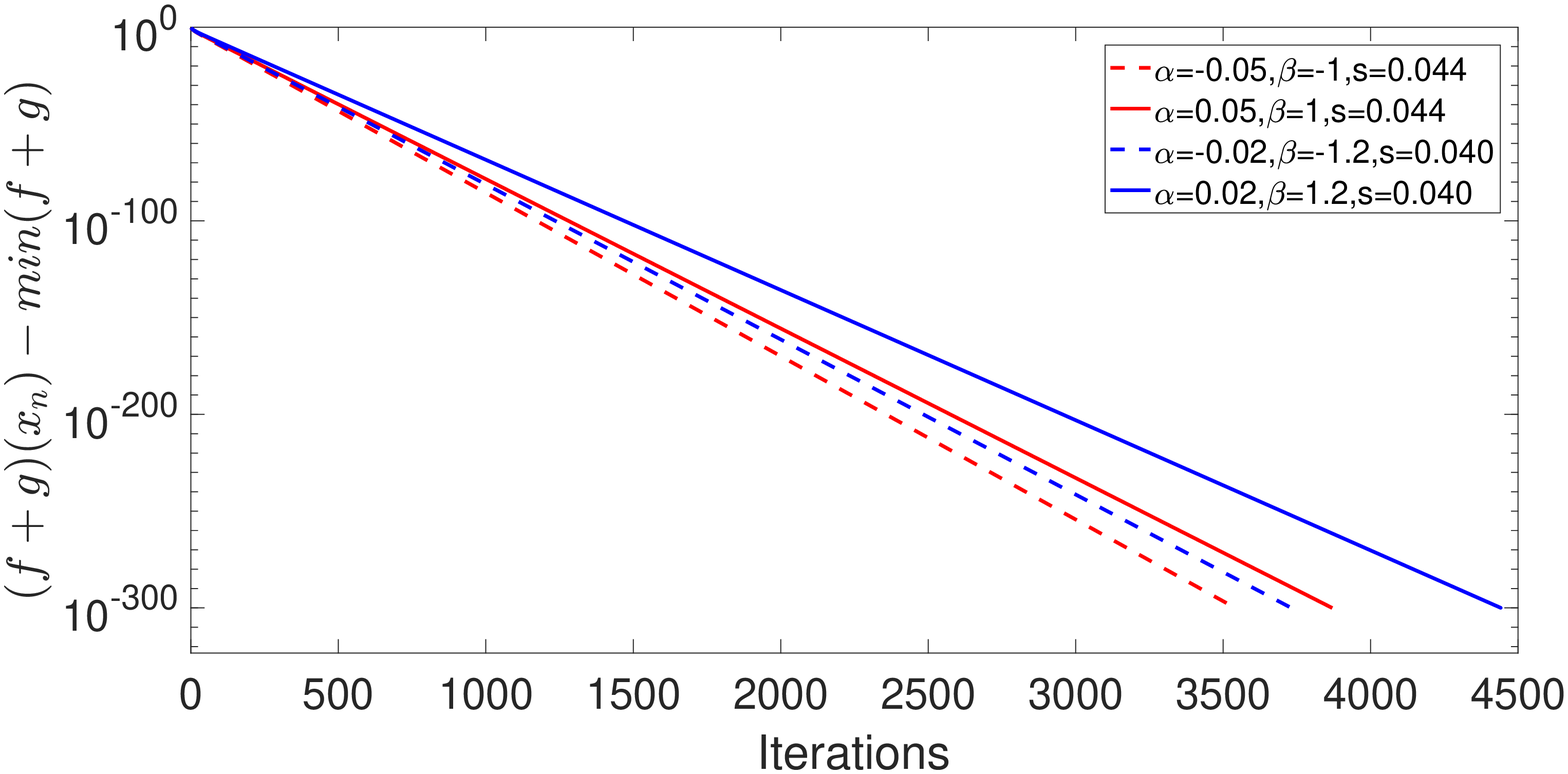}
  \caption{$x_{-1}=x_0=(0.5,-0.5).$}
  \label{fig3:sfig31}
\end{subfigure}
\begin{subfigure}{.5\textwidth}
  \centering
  \includegraphics[width=.99\linewidth]{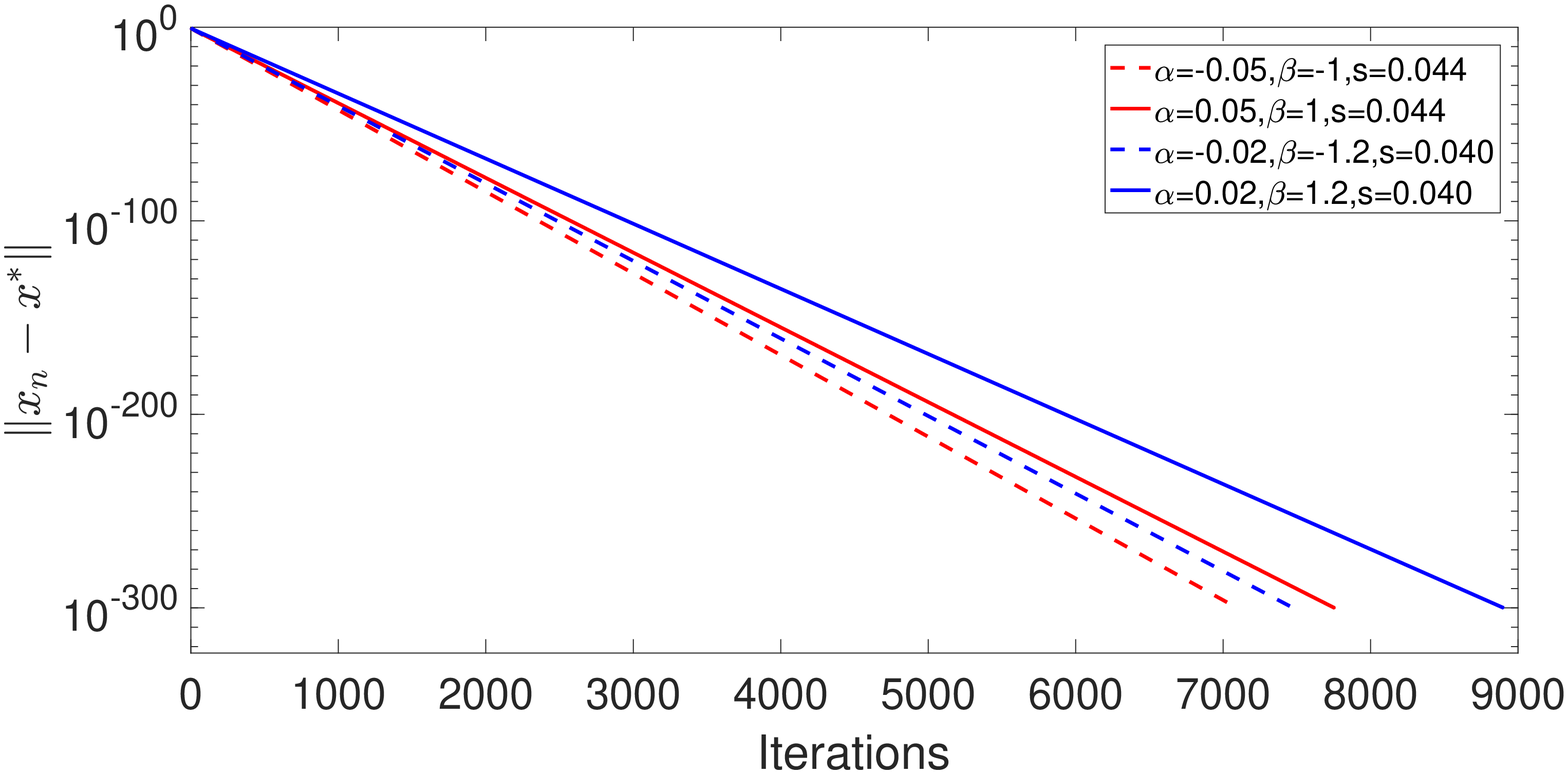}
  \caption{ $x_{-1}=x_0=(0.5,-0.5).$}
  \label{fig3:sfig32}
\end{subfigure}
 \caption{Negative inertial parameters $\a_n$ and $\b_n$ in c-PADISNO may lead to better performance.}
\end{figure}

{\bf C. Better performance of c-PADISNO in the case $\b_n\not\equiv 0$ and in the case $\a_n\not\To1,\,\b_n\not\To 1.$}

In the next experiment, depicted at Figure 4,  we show that the best choice of the inertial parameters in c-PADISNO is not  $\b_n\equiv0$, which corresponds to i-PIANO method, neither $\a=\b=1$  which corresponds to FISTA method. Indeed, note that in FISTA method the inertial parameters are equal and satisfy $\a_n\To 1,\,n\To+\infty.$

 We consider the following instances:
$$(\a,\b,s)\in\left\{\left(\frac{1}{10},\frac15,\frac{7}{100}\right),\left(\frac{1}{2},\frac12,\frac{7}{200}\right),\left(\frac{3}{5},\frac32,\frac{7}{500}\right),
\left(\frac{9}{10},\frac14,\frac{9}{1000}\right), \left(\frac{9}{10},0,\frac{7}{500}\right),\left(1,1,\frac{71}{1000}\right)\right\}.$$

Note that the case $\a=\frac{9}{10},\,\b=0,\,s=\frac{7}{500}$ corresponds to i-PIANO method.  The experiment reveals that  even with worse step size and the same inertial parameter $\a=\frac{9}{10},$  c-PADISNO with $\b\neq 0$ has a better behaviour than the case $\b=0$. Further, if we fix the step size, we can choose $\a$ and $\b$ such that our algorithm with these parameters has a much better behavior than the case $\b=0.$

Further, for comparison purposes, we implemented the case $\a=\b=1,\,s=\frac{71}{1000}=\frac{1}{L_g},$ which corresponds to FISTA method.

\begin{figure}[hbt!]
\begin{subfigure}{.5\textwidth}
  \centering
  \includegraphics[width=.99\linewidth]{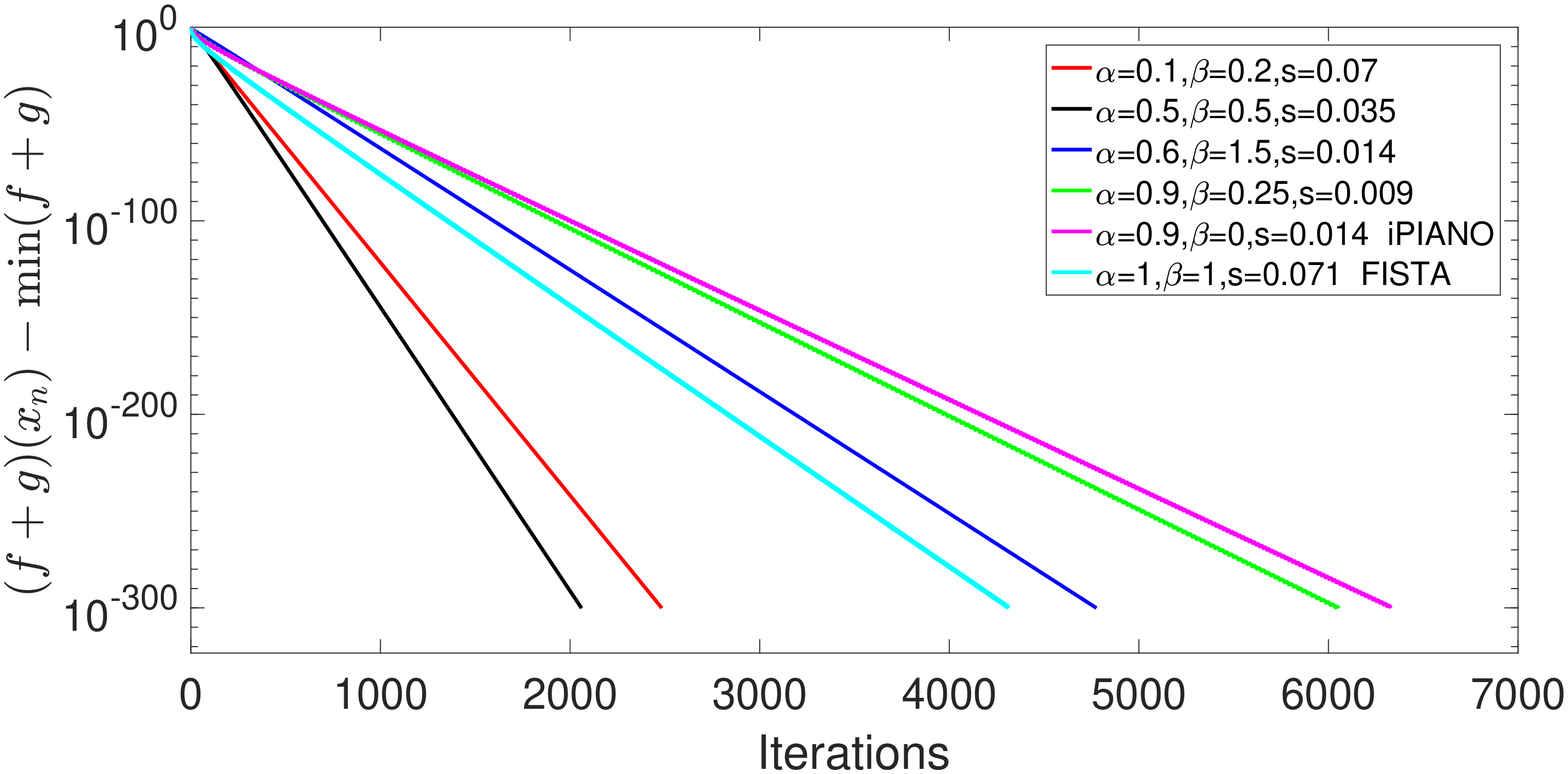}
  \caption{$x_{-1}=x_0=(0.5,-0.5).$}
  \label{fig4:sfig41}
\end{subfigure}
\begin{subfigure}{.5\textwidth}
  \centering
  \includegraphics[width=.99\linewidth]{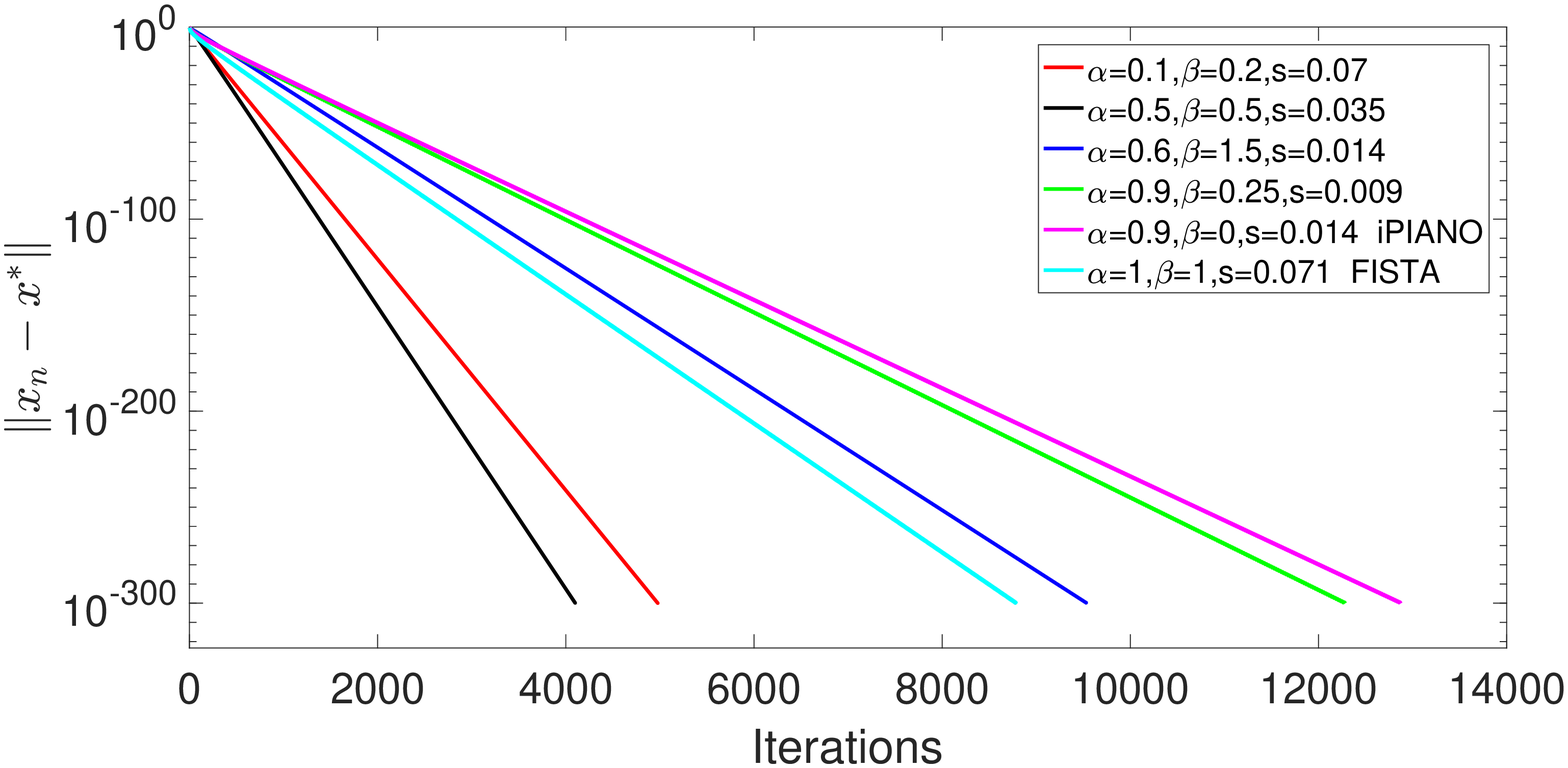}
  \caption{ $x_{-1}=x_0=(0.5,-0.5).$}
  \label{fig4:sfig42}
\end{subfigure}
 \caption{Choosing  $\b\neq 0$ and $\a\neq 1,\,\b\neq 1$ in c-PADISNO lead to better performance.}
\end{figure}

\subsection{The organization of the paper}
The outline of the paper is the following. In the next section we present some notions and preliminary results necessary for carrying out our analysis. In section 3 we state an abstract convergence theorem that can be seen as an extension of the abstract convergence result obtained in \cite{L} in the context of the minimization of a smooth function. Further, by using this abstract result we derive some rates for a sequence generated by an abstract algorithm, in terms of the KL exponent of the abstract objective function. The proofs are postponed to the Appendix.
In section 4 we study the convergence behaviour of the sequences generated by the numerical schemes PADISNO and  c-PADISNO. We prove Theorem \ref{convergence} by showing  that the regularization $H$ in its hypotheses satisfies the assumptions of the abstract convergence theorem obtained in section 3. The Kurdyka-{\L}ojasiewicz property is a key tool in our analysis. We refer the reader also to \cite{attouch-bolte2009}, \cite{att-b-red-soub2010} \cite{CP}, \cite{C}, \cite{FGP}, \cite{BCL} and \cite{OCBP} for literature concerning proximal-gradient splitting methods in the non-convex case relying on the Kurdyka-{\L}ojasiewicz property. In section 5 we prove Theorem \ref{convergencerates} by applying the abstract result concerning convergence rates obtained in section 3. Some important consequences are also discussed. In section 6 we show the usefulness of considering different and also negative inertial parameters in PADISNO, by applying our algorithm to the problem of reconstructing a blurred and noisy image. Finally, we conclude our paper and we outline some avenues of research for the future.

\section{Preliminaries}\label{sec2}

In this section we introduce some basic notions and present preliminary results that will be used in the sequel. The finite-dimensional spaces considered in the  manuscript are endowed with the Euclidean norm topology. The {\it domain} of the function  $f:\R^m\rightarrow \R \cup \{+\infty\}$ is defined by $\dom f=\{x\in\R^m:f(x)<+\infty\}$.
We say that $f$ is {\it proper}, if $\dom f\neq\emptyset$. For the following generalized subdifferential notions and their basic properties we refer to \cite{boris-carte, rock-wets}.
%\subsection{On the limiting subdifferential}

Let $f:\R^m\rightarrow \R \cup \{+\infty\}$ be a proper and lower semicontinuous function. For $x\in\dom f$, the
{\it Fr\'{e}chet (viscosity) subdifferential} of $f$ at $x$ is defined as
$$\hat{\partial}f(x)= \left \{v\in\R^m: \liminf_{y\rightarrow x}\frac{f(y)-f(x)-\<v,y-x\>}{\|y-x\|}\geq 0 \right \}.$$ For
$x\notin\dom f$, we set $\hat{\partial}f(x):=\emptyset$. The {\it limiting (Mordukhovich) subdifferential} is defined at $x\in \dom f$ by
$$\partial f(x)=\{v\in\R^m:\exists x_k\rightarrow x,f(x_k)\rightarrow f(x)\mbox{ and }\exists v_k\in\hat{\partial}f(x_k),v_k\rightarrow v \mbox{ as }k\rightarrow+\infty\},$$
while for $x \notin \dom f$, we set $\partial f(x) :=\emptyset$. It is obvious that $\hat\partial f(x)\subseteq\partial f(x)$ for each $x\in\R^m$.

In case $f$ is convex, these subdifferential notions coincide with the concept of {\it convex subdifferential}, that is
$\hat\partial f(x)=\partial f(x)=\{v\in\R^m:f(y)\geq f(x)+\<v,y-x\> \ \forall y\in \R^m\}$ for all $x\in\dom f$. We recall the well known identity between the proximal point operator of the convex function $f$ and the resolvent operator of its subdifferential $\p f$ that is, the equality
$\prox\nolimits_{f}(x)=(I+\p f)^{-1}(x)$ holds for all $x\in \R^m$, where $I:\R^m\To\R^m$ denotes the identity operator.% and $\p f$ denotes the subdifferential of the convex function $f$.

 A useful property of the graph of the limiting subdifferential is the following closedness criteria : if $(x_k)_{k\geq 0}$ and $(v_k)_{k\geq 0}$ are sequences in $\R^m$ such that
$v_k\in\partial f(x_k)$ for all $k \geq 0$, $(x_k,v_k)\To (x,v)$ and $f(x_k)\rightarrow f(x)$ as $k\rightarrow+\infty$, (obviously $x\in\dom f$), then
$v\in\partial f(x)$.

In this non-smooth setting we have the following Fermat rule: if $x\in\R^m$ is a local minimizer of $f$, then $0\in\partial f(x)$. We underline that in case the function $f$ is continuously differentiable around $x \in \R^m$ we have $\partial f(x)=\{\nabla f(x)\}$. We denote by
$$\crit(f)=\{x\in\R^m: 0\in\partial f(x)\}$$ the set of {\it (limiting)-critical points} of $f$.

We will also need the sum rule that the limitting subdifferential satisfies, that is: if $f:\R^m\To \R \cup \{+\infty\}$ is proper and lower semicontinuous and
$g:\R^m\To \R$ is a continuously differentiable function, then $\partial (f+g)(x)=\partial f(x)+\nabla g(x)$ for all $x\in\R^m$,
hence $$\crit(f+g)=\{x\in\R^m: 0\in\partial f(x)+\n g(x)\}.$$

%\subsection{On KL property}
Another important notion that we need in the sequel is the Kurdyka-\L{}ojasiewicz property of a function.
 %For $\eta\in(0,+\infty]$, we denote by $\Theta_{\eta}$
%the class of concave and continuous functions $\varphi:[0,\eta)\To [0,+\infty)$ such that $\varphi(0)=0$, $\varphi$ is
%continuously differentiable on $(0,\eta)$, continuous at $0$ and $\varphi'(s)>0$ for all
%$s\in(0, \eta)$.
In the following definition (see \cite{att-b-red-soub2010, b-sab-teb}) we use the {\it distance function}
to a set, defined for $A\subseteq\R^m$ as $\dist(x,A)=\inf_{y\in A}\|x-y\|$ for all $x\in\R^m$.

\begin{definition}[\rm Kurdyka-\L{}ojasiewicz property]\label{KL-property}  Let $f:\R^m\To \R \cup \{+\infty\}$ be a proper
and lower semicontinuous function. We say that $f$ satisfies the {\it Kurdyka-\L{}ojasiewicz (KL) property} at
$\ol x\in \dom\partial f=\{x\in\R^m:\partial f(x)\neq\emptyset\}$
if there exist $\eta \in(0,+\infty]$ and a concave and continuous function $\varphi:[0,\eta)\To [0,+\infty)$ such that $\varphi(0)=0$, $\varphi$ is continuously differentiable on $(0,\eta)$ and $\varphi'(s)>0$ for all
$s\in(0, \eta)$, further there exists a neighborhood $U$ of $\ol x$  such that for all $x$ in the
intersection
$$U\cap \{x\in\R^m: f(\ol x)<f(x)<f(\ol x)+\eta\}$$ the following inequality holds
$$\varphi'(f(x)-f(\ol x))\dist(0,\partial f(x))\geq 1.$$
If $f$ satisfies the KL property at each point in $\dom\partial f$, then $f$ is called a {\it KL function}.
\end{definition}
The function $\varphi$ in the above definition is called a desingularizing function \cite{BBJ}.
The origins of this notion go back to the pioneering work of \L{}ojasiewicz \cite{lojasiewicz1963}, where it is proved that for a
real-analytic function $f:\R^m\To\R$ and a critical point $\ol x\in\R^m$ (that is $\nabla f(\ol x)=0$), there exists $\theta\in[1/2,1)$ such that the function
$|f-f(\ol x)|^{\theta}\|\nabla f\|^{-1}$ is bounded around $\ol x$.  In Definition \ref{KL-property} this corresponds to the situation when the desingularizing function has the form
$\varphi(s)=C(1-\theta)^{-1}s^{1-\theta}$. The result of \L{}ojasiewicz allows the interpretation of the KL property as a re-parametrization of the function values in order to avoid flatness around the
critical points. Kurdyka \cite{kurdyka1998} extended this property to differentiable functions definable in an o-minimal structure.
Further extensions to the nonsmooth setting can be found in \cite{b-d-l2006, att-b-red-soub2010, b-d-l-s2007, b-d-l-m2010}.

At first sight the KL property of functions seems restrictive, nevertheless according to \cite{b-sab-teb}, this property is ubiquitous in applications.  To the class of KL functions belong semi-algebraic, globally sub-analytic, uniformly convex and
convex functions satisfying a growth condition. We refer the reader to
\cite{b-d-l2006, att-b-red-soub2010, b-d-l-m2010, b-sab-teb, b-d-l-s2007, att-b-sv2013, attouch-bolte2009} and the references therein  for more details regarding all the classes mentioned above and illustrating examples.

A related notion that  we need is the notion of a KL exponent, which is defined \cite{att-b-red-soub2010,att-b-sv2013,LP} as follows.
\begin{definition}[\rm KL exponent] \label{KLexponent} For a proper closed function $f$ satisfying the KL property at $\ol x\in\dom(\p f)$
if the corresponding  function $\varphi$ can be chosen as $\varphi(t)=\frac{K}{1-\t}t^{1-\t}$ for some $K > 0$ and $\t\in [0,1)$, i.e., there exist
$K, \e > 0$ and $\eta\in (0, +\infty]$ such that
$$\dist(0, \p f(x)) \ge \frac{1}{K}(f(x)-f(\ol x))^\t$$
whenever $\|x - \ol x\|\le\e$ and $f(\ol x) < f(x) < f(\ol x) + \eta$, then we say that $f$ has the KL property at $\ol x$ with an
exponent $\t$. If $f$ is a KL function and has the same exponent $\t$ at any  $\ol x\in\dom(\p f)$, then we say that $f$ is
a KL function with an exponent  $\t.$
\end{definition}

%\section{Convergence}
\section{Abstract convergence results}

In this section we  present several abstract convergence results and also some convergence rates in terms of the KL exponent. The proofs which use some similar techniques as in \cite{att-b-sv2013} are presented at Appendix. For other works where these techniques were used we refer to \cite{FGP,OCBP}.  Our results might become useful in the future for obtaining the convergence of a sequence generated by a forward-backward inertial algorithm in the non-convex setting, where the gradient step is also evaluated in an iteration that contains the inertial term.

In what follows we formulate some conditions that beside the KL property at a point of a proper, lower semi-continuous  function lead to a convergence result.
Consider a sequence $(x_n)_{n\in\N}\subseteq\R^m$ and fix the positive constants $a,b> 0,\,c_1,c_2\ge0,\,c_1^2+c_2^2\neq 0.$ Let $F : \R^m\times\R^m\To \R \cup \{+\infty\}$ be a proper, lower semi-continuous  function. Consider further  a sequence $(u_n)_{n\in\N} := (v_n,w_n)_{n\in\N}\subseteq\R^m\times\R^m$ which is related to the sequence $(x_n)_{n\in\N}$    (with the convention $x_{-1}=x_0\in\R^m$) via the conditions (H1)-(H4) below.
 %Then, the conditions we require for $(u_n)_{n\in\N}$ are:
\vskip0.3cm
(H1) For each $n\in\N$ it holds
$$ a\|x_{n+1}-x_{n}\|^2\le F(u_{n})-F(u_{n+1}).$$

(H2) For each $n \in\N,\,n\ge 1$ one has
$$\dist(0,\p F(u_n))\le b(\|x_n-x_{n-1}\|+\|x_{n-1}-x_{n-2}\|).$$

(H3)   For each $n \in\N$ and every $u=(x,x)\in\R^m\times\R^m$ one has
 $$\|u_n-u\|\le c_1\|x_n-x\|+c_2\|x_{n-1}-x\|.$$

(H4) There exists a subsequence $(u_{n_j} )_{j\in\N}$ of $(u_n)_{n\in\N}$ and $u^*\in\R^m\times\R^m$
 such that
$$u_{n_j}\To u^*\mbox{ and }F(u_{n_j})\To F(u^*),\mbox{ as }j\To+\infty.$$

\begin{remark} One can observe that the conditions (H1) and (H2) are very similar to those in \cite{att-b-sv2013}, \cite{FGP} and \cite{OCBP}, however due to the form of our sequence $(u_n)_{n\in\N}$, there are some major differences. First of all observe that the conditions in \cite{att-b-sv2013} or \cite{FGP} can be rewritten into our setting by considering that the sequence $(u_n)_{n\in\N}$ has the form $u_n=(x_n,x_{n})$ for all $n\in\N$ and the lower semicontinuous function $f$ considered in \cite{att-b-sv2013} satisfies $f(x_n)=F(u_n)$ for all $n\in\N.$ Further, in \cite{OCBP} the sequence $(u_n)_{n\in\N}$ has the special form  $u_n=(x_n,x_{n-1})$  for all $n\in\N.$

Let us mention that in \cite{HLMY} some relaxed versions of (H1) and (H2) were assumed  for a special function $F$ in order to obtain the
linear convergence of an inexact  descent method in a general framework.

\begin{itemize}
\item Our condition (H3) is automatically satisfied for the sequence considered in \cite{att-b-sv2013} that is $u_n=(x_n,x_n)$ with $c_1=\sqrt{2},\,c_2=0$ and also for the sequence considered in \cite{OCBP} $u_n=(x_n,x_{n-1})$ with $c_1=c_2=1.$
\item In \cite{att-b-sv2013} and \cite{FGP} the condition (H1) reads as
$$ a_n\|x_{n+1}-x_{n}\|^2\le F(u_n)-F(u_{n+1}),$$
 where $a_n=a>0$ in \cite{att-b-sv2013} and $a_n>0$ in \cite{FGP}, which are formally  identical to our assumption but our sequence $u_n$ has a more general form, meanwhile in \cite{OCBP} (H1) is
$$ a\|x_n-x_{n-1}\|^2\le F(u_n)-F(u_{n+1}).$$
\item The corresponding relative error (H2) in \cite{att-b-sv2013} is
for each $n\ge 1$ there exists $W_n\in \p F(x_{n+1},x_{n+1})$ such that
$$ \|W_n\|\le b\|x_{n+1}-x_{n}\|$$
consequently, in some sense, our condition may have a larger relative error.
 In \cite{FGP} the condition (H2) has the form
 $$\dist(0,\p F(x_{n+1},x_{n+1}))\le b_n\|x_{n+1}-x_{n}\|+c_n,\mbox{ where }b_n>0,\,c_n\ge 0\mbox{ for all }n\in\N.$$
 Moreover, in \cite{OCBP} is considered $(u_n)_{n\in\N}=(x_n,x_{n-1})_{n\in\N}$, hence their condition (H2) has the form:

 for each $n\ge 1$ there exists $W_n\in \p F(x_{n+1},x_n)$ such that
$$ \|W_n\|\le b(\|x_{n+1}-x_{n}\|+\|x_n-x_{n-1}\|).$$
\item Condition (H4) is identical to condition (H3)  in \cite{att-b-sv2013} and \cite{OCBP}. In \cite{FGP} condition (H3) refers to some properties of the sequences $(a_n)_{n\in\N},(b_n)_{n\in\N}$ and $(c_n)_{n\in\N}.$
\end{itemize}
\end{remark}

\begin{remark} Note that our condition (H2) is equivalent to the following:

(H2') For each $n\ge 1$ there exists $W_n\in \p F(u_n)$ such that
$$ \|W_n\|\le b(\|x_{n}-x_{n-1}\|+\|x_{n-1}-x_{n-2}\|).$$

 Indeed, it is obvious that (H2') implies (H2). Further, since $\p F(u_n)$ is closed and we are in $\R^m$, the projection of $0$ on  $\p F(u_n)$, that is $\pr_{\p F(u_n)}(0)$, is non empty, hence there exists $W_n\in \p F(u_n)$ such that
$$ \dist(0,\p F(u_n))=\|W_n\|\le b(\|x_{n}-x_{n-1}\|+\|x_{n-1}-x_{n-2}\|).$$

However, if one takes instead of $\R^m$ an infinite dimensional Hilbert space $\mathcal{H},$ then condition (H2) is weaker than (H2'). This is due to the fact that in an infinite dimensional Hilbert space $\mathcal{H}$ the set $\pr_{\p F(u_n)}(0)$ might be empty, hence it is  not assured in $\p F(u_n)$  the existence of an element of minimal norm.
\end{remark}

Consequently,  our abstract convergence result stated in Lemma \ref{abstrconv} below is an extension of the corresponding results from \cite{att-b-sv2013}, \cite{FGP} and \cite{OCBP}.

Let us denote by $\omega((x_n)_{n\in\N})$ the set of cluster points of the sequence $(x_n)_{n\in\N}\subseteq\R^m,$ that is,
$$\omega((x_n)_{n\in\N}):=\left\{x^*\in\R^m:\mbox{ there exists a subsequence }(x_{n_j})_{j\in\N}\subseteq(x_n)_{n\in\N}\mbox{ such that }\lim_{j\To+\infty}x_{n_j}=x^*\right\}.$$

\begin{lemma}\label{abstrconv}
Let $F : \R^m\times\R^m\To \R \cup \{+\infty\}$ be a proper and  lower semi-continuous  function which satisfies the Kurdyka-{\L}ojasiewicz property at some point $u^*=(x^*,x^*)\in \R^m\times\R^m.$

Let us denote by $U$, $\eta$ and $\varphi: [0, \eta)\To\R_+$ the objects appearing in the definition of the
KL property at $u^*.$ Let $\sigma>\rho > 0$ be such that $B(u^*,\sigma)\subseteq U.$ Furthermore, consider the sequences $(v_n)_{n\in\N},(w_n)_{n\in\N}$ and assume that the sequences $(x_n)_{n\in\N},\,(u_n)_{n\in\N}=(v_n,w_n)_{n\in\N}\subseteq\R^m\times\R^m$ satisfy the conditions {\rm (H1), (H2)} and {\rm (H3).}

Assume  further that
\begin{equation}\label{e00}
\forall n\in\N : u_n \in B(u^*,\rho)\implies u_{n+1}\in B(u^*,\sigma)\mbox{ with }F(u_{n+1})\ge F(u^*).
\end{equation}
Moreover, the initial points $x_0,u_0$ are such that $u_0\in B(u^*,\rho),$  $F(u^*)\le F(u_0)<F(u^*)+\eta$ and
\begin{equation}\label{e01}
\|x^*-x_0\|+3\sqrt{\frac{F(u_0)-F(u^*)}{a}}+\frac{9b}{4a}\varphi(F(u_0)-F(u^*))<\frac{\rho}{c_1+c_2}.
\end{equation}

Then, the following statements hold.

One has that
$ u_n\in B(u^*,\rho)$ for all $n \in\N.$
Further, $\sum_{n=1}^{+\infty}\|x_n-x_{n-1}\|<+\infty$ and the sequence $(x_n)_{n\in\N}$ converges to a point $\ol x\in\R^m.$ The sequence $(u_n)_{n\in\N}$ converges to  $\ol u= (\ol x,\ol x),$ moreover, we have $\ol u\in B(u^*,\sigma)$ and $F(u_n)\To F(u^*)\ge F(\ol u),\, n \To+\infty.$

Assume further that {\rm (H4)} holds. Then, $\ol u\in\crit F$ and $F(u^*)=F(\ol u).$
 \end{lemma}

\begin{remark} One can observe that our conditions in Lemma \ref{abstrconv} are slightly different  than those in \cite{att-b-sv2013} and \cite{OCBP}. Indeed, we must assume that $u_0\in B(u^*,\rho)$ and  in the  right hand side of \eqref{e01} we have $\frac{\rho}{c_1+c_2}.$
\end{remark}

\begin{corollary}\label{corabstrconv} Assume that the sequences from the definition of $(u_n)_{n\in N}$ have the form $v_n=x_n+\a_n(x_n-x_{n-1})$ and $w_n=x_n+\b_n(x_n-x_{n-1})$  for all $n\ge 1$,  where $(\a_n)_{n\in\N},(\b_n)_{n\in\N}$ are  bounded sequences. Let $c=sup_{n\in\N}(|\a_n|+|\b_n|).$ Then {\rm (H3)} holds with $c_1=2+c$ and $c_2=c$. Further, Lemma \ref{abstrconv} holds true if we replace \eqref{e00} in its hypotheses by
 $$\eta < \frac{a(\sigma-\rho)^2}{4(1+c)^2}\mbox{ and }F(u_n)\ge F(u^*),\mbox{ for all }n\in\N,\,n\ge 1.$$
 \end{corollary}

Now we are ready to formulate the following result.

\begin{theorem}\label{thabstrconv}\rm( Convergence to a critical point.) Let $F : \R^m\times\R^m\To \R \cup \{+\infty\}$ be a proper and lower semi-continuous  function
and  assume that the sequences $(x_n)_{n\in\N},\,(u_n)_{n\in N}=(x_n+\a_n(x_n-x_{n-1}),x_n+\b_n(x_n-x_{n-1}))_{n\in\N}$
satisfy (H1) and (H2), (with the convention $x_{-1}=x_0\in\R^m$), where $(\a_n)_{n\in\N},(\b_n)_{n\in\N}$ are bounded sequences. Moreover, assume that $\omega((u_n)_{n\in\N})$ is nonempty and that $F$ has the Kurdyka-{\L}ojasiewicz property at a point $u^*=(x^*,x^*)\in \omega((u_n)_{n\in\N})$ and for $u^*$ (H4) holds.
Then, the sequence $(x_n)_{n\in\N}$ converges to $x^*$, $(u_n)_{n\in\N}$ converges to $u^*$ and $u^*\in\crit(F).$
\end{theorem}

\begin{remark} As we have emphasized before, the main advantage of the abstract convergence result from this section is that can be applied also for forward-backward algorithms where the gradient step is evaluated in iterations that contain the inertial term. This is due to the fact that the sequence $(u_n)_{n\in\N}$ in the hypotheses of Corollary \ref{corabstrconv} and Theorem \ref{thabstrconv}
is assumed to have the form $(u_n)_{n\in\N}=(x_n+\a_n(x_n-x_{n-1}),x_n+\b_n(x_n-x_{n-1}))_{n\in\N}$.
\end{remark}

In what follows we provide some convergence rates in terms of the KL exponent of $F.$
We have the following result.
\begin{theorem}\label{thabstrrate}\rm(Convergence rates.) Let $F : \R^m\times\R^m\To \R \cup \{+\infty\}$ be a proper and lower semi-continuous  function and assume that the sequences $(x_n)_{n\in\N},\,(u_n)_{n\in N}=(x_n+\a_n(x_n-x_{n-1}),x_n+\b_n(x_n-x_{n-1}))_{n\in\N}\subseteq\R^m$
satisfy (H1) and (H2), (with the convention $x_{-1}=x_0\in\R^m$), where $(\a_n)_{n\in\N},(\b_n)_{n\in\N}\subseteq \R$ are bounded sequences. Moreover, assume that $\omega((u_n)_{n\in\N})$ is nonempty and that $F$ has the Kurdyka-{\L}ojasiewicz property with an exponent $\t\in[0,1)$ at a point $u^*=(x^*,x^*)\in \omega((u_n)_{n\in\N})$ and for $u^*$ (H4) holds.
Then, the sequence $(x_n)_{n\in\N}$ converges to $x^*$, $(u_n)_{n\in\N}$ converges to $u^*$ and $u^*\in\crit(F).$ Further, the following statements hold.
\begin{itemize}
\item[$(\emph{a})$] If $\t=0$ then $(F(u_n))_{n\in\N},(x_n)_{n\in\N}$ and $(u_n)_{n\in\N}$ converge in a finite number of steps;
\item[$(\emph{b})$] If $\t\in\left(0,\frac12\right]$ then there exist $A_1>0,$ $Q\in[0,1)$  and $\ol k\in\N$ such that
$F(u_{n})-F(u^* )\le A_1 Q^n$, $\|x_n-x^*\|\le A_1 Q^{\frac{n}{2}}$ and $\|u_n-u^*\|\le A_1 Q^{\frac{n}{2}}$ for  every $n\ge \ol k$;
\item[$(\emph{c})$] If $\t\in\left(\frac12,1\right)$ then there exist  $A_2>0$ and $\ol k\in\N$ such that
$F(u_n)-F(u^*)\le A_2 {n^{-\frac{1}{2\t-1}}},$ $\|x_n-x^*\|\le A_2 {n^{\frac{\t-1}{2\t-1}}}$  and $\|u_n-u^*\|\le A_2 {n^{\frac{\t-1}{2\t-1}}}\mbox{ for all }n\ge\ol k$.
\end{itemize}
\end{theorem}

\section{The convergence of the numerical schemes PADISNO and c-PADISNO}

In this section we obtain the convergence of the sequences generated by PADISNO and c-PADISNO to a critical point of the objective function $f+g.$ To this purpose we  show that an  appropriate regularization of $f+g$ satisfies the conditions (H1)-(H4) and we apply Theorem \ref{thabstrconv}. An important tool in our forthcoming analysis is the so called descent lemma, see \cite{Nest}, applied to the smooth function $g,$ that is, % More precisely, the descent lemma in our setting reads as
\begin{equation}\label{desc}g(y)\le g(x)+\<\n g(x),y-x\>+\frac{L_g}{2}\|y-x\|^2,\,\forall x,y\in\R^m.
\end{equation}

Now we are able to obtain a descent property using the iterates generated by PADISNO, more precisely the following result holds.

\begin{lemma}\label{descentnc}  In the settings of problem \eqref{propt}, for some starting points $x_{-1}=x_0\in\R^m,$ let $(x_n)_{n\in\N}$ be the sequence generated by the numerical scheme PADISNO. Consider the sequence
$$\d_n=\frac{1}{4s}+\frac{L_g}{4}(|\b_n|-|\b_{n-1}|-1),\,n\ge 1.$$
Then, there exist $N\in\N$ and $A>0$ such that
\begin{itemize}
\item[(i)] $\left((f+g)(x_{n})+\d_{n}\|x_{n}-x_{n-1}\|^2\right)-\left((f+g)(x_{n+1})+\d_{n+1}\|x_{n+1}-x_{n}\|^2\right)\ge A\|x_{n+1}-x_n\|^2$ and $\d_n>0$  for all $n\ge N$.
\end{itemize}
Assume that $f+g$ is bounded from below. Then, the following statements hold.
\begin{itemize}
\item[(ii)] The sequences $\left((f+g)(x_{n})+\d_{n}\|x_{n}-x_{n-1}\|^2\right)_{n\in \N}$ and $\left((f+g)(x_{n})\right)_{n\in \N}$ are convergent;
\item[(iii)] $\sum_{n\ge 1}\|x_n-x_{n-1}\|^2<+\infty.$
\end{itemize}
\end{lemma}
\begin{proof}
 From PADISNO we have
 $$f(x_{n+1})+\<\n g(z_n), x_{n+1}-y_n\>+\frac{1}{2s}\|x_{n+1}-y_n\|^2\le f(x_{n})+\<\n g(z_n), x_{n}-y_n\>+\frac{1}{2s}\|x_{n}-y_n\|^2,\mbox{ for all }n\in\N.$$
 In other words
 \begin{equation}\label{nc0}
 f(x_{n})-f(x_{n+1})\ge \<\n g(z_n), x_{n+1}-x_n\>+\frac{1}{2s}\|x_{n+1}-x_n\|^2-\frac{\a_n}{s}\<x_{n+1}-x_n,x_n-x_{n-1}\>,\mbox{ for all }n\in\N.
 \end{equation}

Further,
\begin{equation}\label{ng0}
\<\n g(z_n), x_{n+1}-x_n\>=\<\n g(z_n)-\n g(x_n),x_{n+1}-x_n\>+\<\n g(x_n),x_{n+1}-x_n\>,\mbox{ for all }n\in\N.
\end{equation}

By using the Lipschitz continuity of $\n g$  we get
\begin{align}\nonumber
\<\n g(z_n)-\n g(x_n),x_{n+1}-x_n\>&=\left\|\frac{1}{\sqrt{2L_g|\b_n|}}(\n g(z_n)-\n g(x_n))+\frac{\sqrt{2L_g|\b_n|}}{2}(x_{n+1}-x_n)\right\|^2\\
\nonumber&-\frac{1}{2L_g|\b_n|}\|\n g(z_n)-\n g(x_n)\|^2-\frac{L_g|\b_n|}{2}\|x_{n+1}-x_n\|^2\\
\nonumber&\ge -\frac{L_g|\b_n|}{2}(\|x_n-x_{n-1}\|^2+\|x_{n+1}-x_n\|^2),
\end{align}
for all $n\in\N$ for which for $\b_n\neq 0$. Further, if $\b_n=0$ then the above inequality trivially holds.

By using \eqref{desc} we get
$$\<\n g(x_n),x_{n+1}-x_n\>\ge g(x_{n+1})-g(x_n)-\frac{L_g}{2}\|x_{n+1}-x_n\|^2,\mbox{ for all }n\in\N.$$

Hence, \eqref{ng0} leads to
\begin{equation}\label{ng1}
\<\n g(z_n), x_{n+1}-x_n\>\ge g(x_{n+1})-g(x_n)-\frac{L_g}{2}(1+|\b_n|)\|x_{n+1}-x_n\|^2-\frac{L_g|\b_n|}{2}\|x_n-x_{n-1}\|^2,\mbox{ for all }n\in\N.
\end{equation}

Hence, \eqref{nc0} leads to
\begin{align}\label{fg1}
(f+g)(x_n)-(f+g)(x_{n+1})&\ge \left(\frac{1}{2s}-\frac{L_g}{2}(1+|\b_n|)\right)\|x_{n+1}-x_n\|^2-\frac{\a_n}{s}\<x_{n+1}-x_n,x_n-x_{n-1}\>\\
\nonumber&-\frac{L_g|\b_n|}{2}\|x_n-x_{n-1}\|^2,\mbox{ for all }n\in\N.
\end{align}

Now, taking into account the form of $\d_n$, \eqref{fg1} leads to
\begin{align}\label{fg2}
&\left((f+g)(x_n)+\d_n\|x_n-x_{n-1}\|^2\right)-\left((f+g)(x_{n+1})+\d_{n+1}\|x_{n+1}-x_n\|^2\right)\ge \\ \nonumber &\left(\frac{1}{2s}-\frac{L_g}{2}(1+|\b_n|)-\d_{n+1}\right)\|x_{n+1}-x_n\|^2-\frac{\a_n}{s}\<x_{n+1}-x_n,x_n-x_{n-1}\>+\left(\d_n-\frac{L_g|\b_n|}{2}\right)\|x_n-x_{n-1}\|^2\\
\nonumber&=\left(\frac{1}{4s}-\frac{L_g}{4}(1+|\b_n|+|\b_{n+1}|)\right)\|x_{n+1}-x_n\|^2\\
\nonumber&-\frac{\a_n}{s}\<x_{n+1}-x_n,x_n-x_{n-1}\>\\
\nonumber&+\left(\frac{1}{4s}-\frac{L_g}{4}(1+|\b_{n-1}|+|\b_{n}|)\right)\|x_{n}-x_{n-1}\|^2
\end{align}
for all $n\in\N.$

For simplicity, let us denote $B_n=\frac{1}{4s}-\frac{L_g}{4}(1+|\b_{n-1}|+|\b_{n}|),\,n\in\N.$ Note that by assumption we have $0<s<\frac{1-2|\a|}{L_g(2|\b|+1)}$ and $\a\in\left(-\frac{1}{2},\frac{1}{2}\right),$ hence
$$\lim_{n\To+\infty}B_n=\frac{1}{4s}-\frac{L_g}{4}(1+2|\b|)>0.$$
Consequently, there exists $N_0\in\N$ such that $B_n>0$ for all $n\ge N_0.$

Thus, we have
\begin{align}\label{forB}
-\frac{\a_n}{s}\<x_{n+1}-x_n,x_n-x_{n-1}\>&=\left\|-\frac{\a_n}{s}\frac{1}{2\sqrt{B_n}}(x_{n+1}-x_n)+\sqrt{B_n}(x_n-x_{n-1})\right\|^2\\
\nonumber&-\frac{\a_n^2}{4s^2 B_n}\|x_{n+1}-x_n\|^2-B_n\|x_n-x_{n-1}\|^2,
\end{align}
for all $n\ge N_0.$

Consequently, \eqref{fg2} leads to
\begin{align}\label{fg3}
&\left((f+g)(x_n)+\d_n\|x_n-x_{n-1}\|^2\right)-\left((f+g)(x_{n+1})+\d_{n+1}\|x_{n+1}-x_n\|^2\right)\ge \\
\nonumber &\left(B_{n+1}-\frac{\a_n^2}{4s^2 B_n}\right)\|x_{n+1}-x_n\|^2,\mbox{ for all }n\ge N_0.
\end{align}

Further, we have that
$$\lim_{n\To+\infty}\left(B_{n+1}-\frac{\a_n^2}{4s^2 B_n}\right)=\frac{\left(\frac{1-2|\a|}{4s}-\frac{L_g}{4}(1+2|\b|)\right)\left(\frac{1+2|\a|}{4s}-\frac{L_g}{4}(1+2|\b|)\right)}{\frac{1}{4s}-\frac{L_g}{4}(1+2|\b|)}>0,$$
since $0<s<\frac{1-2|\a|}{L_g(2|\b|+1)}$ and $\a\in\left(-\frac{1}{2},\frac{1}{2}\right),$ hence there exists $N_1\ge N_0$ and $A>0$ such that
$$B_{n+1}-\frac{\a_n^2}{4s^2 B_n}\ge A,\mbox{ for all }n\ge N_1.$$

Finally, $\lim_{n\To+\infty}\d_n=\frac{1}{4s}-\frac{L_g}{4}>0$, hence there exists $N\ge N_1$ such that $\d_n>0$ for all $n\ge N.$

In other words, for all $n\ge N$ one has
\begin{align}\label{fg4}
&\left((f+g)(x_n)+\d_n\|x_n-x_{n-1}\|^2\right)-\left((f+g)(x_{n+1})+\d_{n+1}\|x_{n+1}-x_n\|^2\right)\ge A\|x_{n+1}-x_n\|^2
\end{align}
and $\d_n>0$ and this proves (i).

Let $r>N.$ By summing up the \eqref{fg4} from $n=N$ to $n=r$ we get
$$
A\sum_{n=N}^r\|x_{n+1}-x_{n}\|^2\le ((f+g)(x_N)+\d_N\|x_N-x_{N-1}\|^2)-((f+g)(x_{r+1})+\d_{r+1}\|x_{r+1}-x_r\|^2)
$$
which leads to
\begin{equation}\label{forcoercive}
(f+g)(x_{r+1})+A\sum_{n=N}^r\|x_{n+1}-x_{n}\|^2\le (f+g)(x_N)+\d_N\|x_N-x_{N-1}\|^2.
\end{equation}
Now,  if we assume that  $f+g$ is bounded from below, by letting $r\To+\infty$  we obtain
$$\sum_{n=N}^{\infty}\|x_{n+1}-x_{n}\|^2<+\infty$$ which proves (iii).

The latter relation also shows that
$$\lim_{n\To+\infty}\|x_n-x_{n-1}\|^2=0,$$
hence
\begin{equation}\label{forcont}
\lim_{n\To+\infty}\d_n\|x_n-x_{n-1}\|^2=0.
\end{equation}
But then, by using the  assumption that the function $f+g$ is bounded from below we obtain that the sequence $((f+g)(x_n)+\d_n\|x_n-x_{n-1}\|^2)_{n\in\N}$ is bounded from below. On the other hand, from (i) we have that  the sequence $((f+g)(x_n)+\d_n\|x_n-x_{n-1}\|^2)_{n\ge N}$ is nonincreasing, hence there exists
$$\lim_{n\To+\infty} (f+g)(x_n)+\d_n\|x_n-x_{n-1}\|^2\in\R.$$
Further, since $\lim_{n\To+\infty}\d_n\|x_n-x_{n-1}\|^2=0$ we get that
there exists
$$\lim_{n\To+\infty} (f+g)(x_n)\in\R.$$
\end{proof}

If we assume that the function $g$ is concave the previous result can be improved.

\begin{lemma}\label{descentncgconcave}  Assume that the function $g$, in the formulation of the optimization problem \eqref{propt}, is concave. For some starting points $x_{-1}=x_0\in\R^m,$ let $(x_n)_{n\in\N}$ be the sequence generated by the numerical scheme PADISNO. Assume further, that the step size $s$ in PADISNO satisfies $0<s<\frac{1-2|\a|}{2Lg|\b|},$ (when $\b=0$ we assume only that $s\in(0,+\infty)$). Consider the sequence
$$\d_n=\frac{1}{4s}+\frac{L_g}{4}(|\b_n|-|\b_{n-1}|),\,n\ge 1.$$
Then, there exist $N\in\N$ and $A>0$ such that
\begin{itemize}
\item[(i)] $\left((f+g)(x_{n})+\d_{n}\|x_{n}-x_{n-1}\|^2\right)-\left((f+g)(x_{n+1})+\d_{n+1}\|x_{n+1}-x_{n}\|^2\right)\ge A\|x_{n+1}-x_n\|^2$ and $\d_n>0$  for all $n\ge N$.
\end{itemize}
Assume that $f+g$ is bounded from below. Then, the following statements hold.
\begin{itemize}
\item[(ii)] The sequences $\left((f+g)(x_{n})+\d_{n}\|x_{n}-x_{n-1}\|^2\right)_{n\in \N}$ and $\left((f+g)(x_{n})\right)_{n\in \N}$ are convergent;
\item[(iii)] $\sum_{n\ge 1}\|x_n-x_{n-1}\|^2<+\infty.$
\end{itemize}
\end{lemma}
\begin{proof} We argue as in Lemma \ref{descentnc} but instead of descent lemma we use the fact that $g$  is concave, hence the gradient inequality yields
$$\<\n g(x_n),x_{n+1}-x_n\>\ge g(x_{n+1})-g(x_n),\mbox{ for all }n\in\N.$$

Consequently,  in this case \eqref{nc0} leads to
\begin{align}\label{fg1gconcave}
(f+g)(x_n)-(f+g)(x_{n+1})&\ge \left(\frac{1}{2s}-\frac{L_g|\b_n|}{2})\right)\|x_{n+1}-x_n\|^2-\frac{\a_n}{s}\<x_{n+1}-x_n,x_n-x_{n-1}\>\\
\nonumber&-\frac{L_g|\b_n|}{2}\|x_n-x_{n-1}\|^2,\mbox{ for all }n\in\N.
\end{align}

Now, taking into account the form of $\d_n$, \eqref{fg1gconcave} leads to
\begin{align}\label{fg2gconcave}
&\left((f+g)(x_n)+\d_n\|x_n-x_{n-1}\|^2\right)-\left((f+g)(x_{n+1})+\d_{n+1}\|x_{n+1}-x_n\|^2\right)\ge \\ \nonumber &\left(\frac{1}{2s}-\frac{L_g|\b_n|}{2}-\d_{n+1}\right)\|x_{n+1}-x_n\|^2-\frac{\a_n}{s}\<x_{n+1}-x_n,x_n-x_{n-1}\>+\left(\d_n-\frac{L_g|\b_n|}{2}\right)\|x_n-x_{n-1}\|^2\\
\nonumber&=\left(\frac{1}{4s}-\frac{L_g}{4}(|\b_n|+|\b_{n+1}|)\right)\|x_{n+1}-x_n\|^2\\
\nonumber&-\frac{\a_n}{s}\<x_{n+1}-x_n,x_n-x_{n-1}\>\\
\nonumber&+\left(\frac{1}{4s}-\frac{L_g}{4}(|\b_{n-1}|+|\b_{n}|)\right)\|x_{n}-x_{n-1}\|^2
\end{align}
for all $n\in\N.$

For simplicity, let us denote $B_n=\frac{1}{4s}-\frac{L_g}{4}(|\b_{n-1}|+|\b_{n}|),\,n\in\N.$

Note that by assumption we have $0<s<\frac{1-2|\a|}{2L_g|\b|}\le \frac{1}{2L_g|\b|}$, hence
$$\lim_{n\To+\infty}B_n=\frac{1}{4s}-\frac{L_g|\b|}{2}>0.$$
Consequently, there exists $N_0\in\N$ such that $B_n>0$ for all $n\ge N_0.$

Thus, using  \eqref{forB} combined with \eqref{fg2gconcave} we get that for all $n\ge N_0$ it holds
\begin{align}\label{fg3gconcave}
&\left((f+g)(x_n)+\d_n\|x_n-x_{n-1}\|^2\right)-\left((f+g)(x_{n+1})+\d_{n+1}\|x_{n+1}-x_n\|^2\right)\ge
\left(B_{n+1}-\frac{\a_n^2}{4s^2 B_n}\right)\|x_{n+1}-x_n\|^2.
\end{align}

Further, we have that
$$\lim_{n\To+\infty}\left(B_{n+1}-\frac{\a_n^2}{4s^2 B_n}\right)=\frac{\left(\frac{1-2|\a|}{4s}-\frac{L_g|\b|}{2}\right)\left(\frac{1+2|\a|}{4s}-\frac{L_g|\b|}{2}\right)}{\frac{1}{4s}-\frac{L_g|\b|}{2}}>0,$$
since $0<s<\frac{1-2|\a|}{2L_g|\b|}$ and $\a\in\left(-\frac{1}{2},\frac{1}{2}\right),$ hence there exists $N_1\ge N_0$ and $A>0$ such that
$$B_{n+1}-\frac{\a_n^2}{4s^2 B_n}\ge A,\mbox{ for all }n\ge N_1.$$

Just as in the proof of Lemma \ref{descentnc} we conclude that there exists $N\ge N_1$ such that for all $n\ge N$ one has
\begin{align}\label{fg4gconcave}
&\left((f+g)(x_n)+\d_n\|x_n-x_{n-1}\|^2\right)-\left((f+g)(x_{n+1})+\d_{n+1}\|x_{n+1}-x_n\|^2\right)\ge A\|x_{n+1}-x_n\|^2
\end{align}
and $\d_n>0$ and this proves (i).

The claims (ii) and (iii)  can be proven similarly to the proof of Lemma \ref{descentnc}.
\end{proof}

Concerning c-PADISNO some similar descent property as that stated in Lemma \ref{descentnc} can be obtained. 
%Now we are able to obtain a decrease property for the iterates generated by \eqref{fist}.

\begin{lemma}\label{descent}  In the settings of problem \eqref{propt}, for some starting points $x_{-1}=x_0\in\R^m,$ let $(x_n)_{n\in\N}$ be the sequence generated by the numerical scheme c-PADISNO. Consider the sequence
$$\d_n=\frac{1}{2s}+\frac{L_g}{4}(|\b_n|-|\b_{n-1}|-1),\,n\ge 1.$$
Then, there exists $N\in\N$ and $A>0$ such that
\begin{itemize}
\item[(i)] $\left((f+g)(x_{n})+\d_{n}\|x_{n}-x_{n-1}\|^2\right)-\left((f+g)(x_{n+1})+\d_{n+1}\|x_{n+1}-x_{n}\|^2\right)\ge A\|x_{n+1}-x_n\|^2$ and $\d_n>0$  for all $n\ge N$.
\end{itemize}
Assume that $f+g$ is bounded from below. Then, the following statements hold.
\begin{itemize}
\item[(ii)] The sequences $\left((f+g)(x_{n})+\d_{n}\|x_{n}-x_{n-1}\|^2\right)_{n\in \N}$ and $\left((f+g)(x_{n})\right)_{n\in \N}$ are convergent;
\item[(iii)] $\sum_{n\ge 1}\|x_n-x_{n-1}\|^2<+\infty.$
\end{itemize}
\end{lemma}
\begin{proof}
 From c-PADISNO we have
 $$\frac{1}{s}(y_n-x_{n+1})-\n g(z_n)\in \p f(x_{n+1}),\mbox{ for all }n\in\N.$$
From sub-gradient inequality applied to $f$ we get
\begin{equation}\label{fg01}
f(x_n)-f(x_{n+1})\ge \left\<\frac{1}{s}(y_n-x_{n+1})-\n g(z_n),x_n-x_{n+1}\right\>,\mbox{ for all }n\in\N.
\end{equation}

According to \eqref{ng1} one has
\begin{equation}\label{ng11}
\<-\n g(z_n),x_n-x_{n+1}\>\ge g(x_{n+1})-g(x_n)-\frac{L_g}{2}(1+|\b_n|)\|x_{n+1}-x_n\|^2-\frac{L_g|\b_n|}{2}\|x_n-x_{n-1}\|^2,\mbox{ for all }n\in\N.
\end{equation}

It is straightforward that
$$\left\<\frac{1}{s}(y_n-x_{n+1}),x_n-x_{n+1}\right\>=\frac{1}{s}\|x_{n+1}-x_n\|^2-\frac{\a_n}{s}\<x_{n+1}-x_n,x_n-x_{n-1}\>,\mbox{ for all }n\in\N,$$
hence, \eqref{fg01} leads to
\begin{align}\label{fg11}
(f+g)(x_n)-(f+g)(x_{n+1})&\ge \left(\frac{1}{s}-\frac{L_g}{2}(1+|\b_n|)\right)\|x_{n+1}-x_n\|^2-\frac{\a_n}{s}\<x_{n+1}-x_n,x_n-x_{n-1}\>\\
\nonumber&-\frac{L_g|\b_n|}{2}\|x_n-x_{n-1}\|^2,\mbox{ for all }n\in\N.
\end{align}

The rest of the proof goes analogously to the proof of Lemma \ref{descentnc} and therefore we omit it.
\end{proof}

In case $g$ is concave the objective function in the optimization problem \eqref{propt} becomes the difference of two convex functions. Also in this case one may take very large step size, provided $\b_n\To0,\,n\To+\infty.$ We have the following result.

\begin{lemma}\label{descentgconcave}  Assume that the function $g$, in the formulation of the optimization problem \eqref{propt}, is concave. For some starting points $x_{-1}=x_0\in\R^m,$ let $(x_n)_{n\in\N}$ be the sequence generated by the numerical scheme c-PADISNO. Assume further, that the step size $s$ in c-PADISNO satisfies $0<s<\frac{1-|\a|}{Lg|\b|},$ (when $\b=0$ we assume only that $s\in(0,+\infty)$). Consider the sequence
$$\d_n=\frac{1}{2s}+\frac{L_g}{4}(|\b_n|-|\b_{n-1}|),\,n\ge 1.$$
Then, there exist $N\in\N$ and $A>0$ such that
\begin{itemize}
\item[(i)] $\left((f+g)(x_{n})+\d_{n}\|x_{n}-x_{n-1}\|^2\right)-\left((f+g)(x_{n+1})+\d_{n+1}\|x_{n+1}-x_{n}\|^2\right)\ge A\|x_{n+1}-x_n\|^2$ and $\d_n>0$  for all $n\ge N$.
\end{itemize}
Assume that $f+g$ is bounded from below. Then, the following statements hold.
\begin{itemize}
\item[(ii)] The sequences $\left((f+g)(x_{n})+\d_{n}\|x_{n}-x_{n-1}\|^2\right)_{n\in \N}$ and $\left((f+g)(x_{n})\right)_{n\in \N}$ are convergent;
\item[(iii)] $\sum_{n\ge 1}\|x_n-x_{n-1}\|^2<+\infty.$
\end{itemize}
\end{lemma}
\begin{proof} We argue as in Lemma \ref{descent} but, as in the proof of Lemma \ref{descentncgconcave}, instead of descent lemma we use the the gradient inequality.%$$\<\n g(x_n),x_{n+1}-x_n\>\ge g(x_{n+1})-g(x_n),\mbox{ for all }n\in\N.$$
\end{proof}

\begin{remark} Note that in Lemma \ref{descentncgconcave} and Lemma \ref{descentgconcave} one may allow arbitrary large step size. Indeed, for an arbitrary $K>0$ and a fixed $(\a_n)$ satisfying $\a_n\To\a\in\left(-\frac12,\frac12\right),\,n\To+\infty$ in case of the algorithm PADISNO and $\a_n\To\a\in(-1,1),\,n\To+\infty$ in case of the algorithm c-PADISNO one may consider
$$\b_n=\frac{1-2|\a_n|}{2L_g K}\mbox{ and }\b_n=\frac{1-|\a_n|}{L_g K}, \mbox{ respectively.}$$
Then, our assumption becomes $0<s<K,$ which shows that the step size can be arbitrary large. 

We emphasize that, though in our forthcoming results we do not treat  separately the case when the function $g$ is concave,  in that case the conclusions of Theorem \ref{convergence} and Theorem \ref{convergencerates} hold whenever  we assume that in PADISNO the step size satisfies $0<s<\frac{1-2|\a|}{2L|\b|}$ and in c-PADISNO  the step size satisfies $0<s<\frac{1-|\a|}{L|\b|}$, respectively.

This is due to the fact that in the forthcoming proofs we will not use the conditions imposed on the step size $s$, but the descent properties obtained in Lemma \ref{descentnc}, Lemma \ref{descentncgconcave}, Lemma \ref{descent} and Lemma \ref{descentgconcave}.
\end{remark}

In what follows, in order to apply our abstract convergence result obtained in Theorem \ref{thabstrconv}, we introduce a function and a sequence that will play the role of the function $F$ and the sequence $(u_n)$ studied in the previous section. We will treat PADISNO and c-PADISNO simultaneously, the entities $\d_n$ and $N$ correspond to the appropriate values obtained in Lemma \ref{descentnc} and Lemma \ref{descent}, respectively.

 Consider  the sequence
$$w_n=\sqrt{2\d_n}(x_{n}-x_{n-1})+x_n\mbox{, for all }n\in\N,\,n\ge N$$ and the sequence $u_n=(x_{n+N},w_{n+N})$ for all $n\in\N,$
where $N$ and $\d_n$ were defined in Lemma \ref{descentnc} if $x_n$ is the sequence generated by PADISNO and  $N$ and $\d_n$ were defined in Lemma \ref{descent} if $x_n$ is the sequence generated by c-PADISNO. Let us introduce the following notations:
 $$\tilde{x}_n=x_{n+N},\,\tilde{y}_n=y_{n+N}\mbox{ and }\tilde{z}_n=z_{n+N},$$
$$\tilde{\a}_n=\a_{n+N},\,\tilde{\b}_n=\b_{n+N}\mbox{ and }\tilde{\d}_n=\sqrt{2\d_{n+N}},$$
  for all $n\in\N.$
  Then obviously the sequences $(\tilde{\a}_n)_{n\in\N},(\tilde{\b}_n)_{n\in\N}$ and $(\tilde{\d}_n)_{n\in\N}$ are bounded, (actually they are convergent), and for each $n\in\N$, the sequence $u_n$ has the form
  \begin{equation}\label{zform}
u_n=\left(\tilde{x}_{n},\tilde{x}_{n}+\tilde{\d}_n(\tilde{x}_{n}-\tilde{x}_{n-1})\right).
\end{equation}

Consider further the following regularization of $f+g$
$$H:\R^m\times\R^m\To\R\cup\{+\infty\},\,H(x,y)=(f+g)(x)+\frac12\|y-x\|^2.$$
We have the following result.
\begin{proposition}\label{condH}
The sequences $(\tilde{x}_n)_{n\in\N}$ and $(u_n)_{n\in\N}$ and the function $H$ satisfy the conditions (H1)-(H3).
\end{proposition}
\begin{proof} Indeed, note  first that the sequence $u_n$ has the form assumed at Corollary \ref{corabstrconv} and Theorem \ref{thabstrconv}, hence in particular (H3) holds.

Further, for every $n\in\N$ one has
$$H(u_n)=(f+g)(\tilde{x}_n)+\d_{n+N}\|\tilde{x}_n-\tilde{x}_{n-1}\|^2.$$

Now, (i) from Lemma \ref{descentnc} or Lemma \ref{descent} becomes
\begin{equation}\label{H1forH}
A\|\tilde{x}_{n+1}-\tilde{x}_{n}\|\le H(u_n)-H(u_{n+1}),\,\mbox{ for all }n\in\N,
\end{equation}
which is exactly our condition (H1) applied to the function $H$ and the sequences $(\tilde{x}_n)_{n\in\N}$ and $(u_n)_{n\in\N}.$

Observe that $$\p H(x,y)=\{\p f(x)+\n g(x)+x-y\}\times\{y-x\}.$$
Obviously
$$\p H(u_n)=\{\p f(\tilde{x}_n)+\n g(\tilde{x}_n)-\tilde{\d}_n(\tilde{x}_{n}-\tilde{x}_{n-1})\}\times\{\tilde{\d}_n(\tilde{x}_{n}-\tilde{x}_{n-1})\}$$
for all $n\ge 1.$

From PADISNO and also from c-PADISNO we have $\frac{1}{s}(y_{n-1}-x_{n})-\n g(z_{n-1})\in\p f({x}_n)$ for all $n\ge 1$,
hence
$$\frac{1}{s}(\tilde{y}_{n-1}-\tilde{x}_{n})-\n g(\tilde{z}_{n-1})\in\p f(\tilde{x}_n),$$ for all $n\ge 1.$

Consequently,
$$W_n=\left(\frac{1}{s}(\tilde{y}_{n-1}-\tilde{x}_{n})-\n g(\tilde{z}_{n-1})+\n g(\tilde{x}_n)-\tilde{\d}_n(\tilde{x}_{n}-\tilde{x}_{n-1}),\tilde{\d}_n(\tilde{x}_{n}-\tilde{x}_{n-1})\right)\in \p H(u_n),$$
for all $n\ge 1.$

We have, for all $n\ge 1$ that
\begin{align}\label{fordif}
\|W_n\|&=\sqrt{\left\|\frac{1}{s}(\tilde{y}_{n-1}-\tilde{x}_{n})-\n g(\tilde{z}_{n-1})+\n g(\tilde{x}_n)-\tilde{\d}_n(\tilde{x}_{n}-\tilde{x}_{n-1})\right\|^2+\|\tilde{\d}_n(\tilde{x}_{n}-\tilde{x}_{n-1})\|^2}\\
\nonumber&\le\sqrt{\frac{2}{s^2}\|\tilde{y}_{n-1}-\tilde{x}_{n}\|^2+2L_g^2\|\tilde{z}_{n-1}-\tilde{x}_{n}\|^2+4\tilde{\d}_n^2\|\tilde{x}_{n}-\tilde{x}_{n-1}\|^2}\\
\nonumber&\le\sqrt{\left(\frac{4}{s^2}+4L_g^2+4\tilde{\d}_n^2\right)\|\tilde{x}_{n}-\tilde{x}_{n-1}\|^2+\left(\frac{4\tilde{\a}_{n-1}^2}{s^2}+4L_g^2\tilde{\b}_{n-1}^2\right)\|\tilde{x}_{n-1}-\tilde{x}_{n-2}\|^2}\\
\nonumber&\le b(\|\tilde{x}_{n}-\tilde{x}_{n-1}\|+\|\tilde{x}_{n-1}-\tilde{x}_{n-2}\|),
\end{align}
where $b=\sqrt{\max\left(\frac{4}{s^2}+4L_g^2+4\tilde{\d}_n^2,\frac{4\tilde{\a}_{n-1}^2}{s^2}+4L_g^2\tilde{\b}_{n-1}^2\right)}.$
On the other hand
$\dist(0,\p H(u_n))\le\|W_n\|$ for all $n\ge 1$, which combined with \eqref{fordif} gives (H2).
\end{proof}

Next we present some results concerning the limit points of the sequences $(x_n)_{n\in\N}$ and $(u_n)_{n\in\N}$ and the critical points of the functions $f+g$ and $H$, respectively.

\begin{lemma}\label{regularization} In the settings of problem \eqref{propt}, for some starting points $x_{-1}=x_0\in\R^m,$ consider the sequences $(x_n)_{n\in\N},\,(y_n)_{n\in\N}$ generated by PADISNO or c-PADISNO. Assume that $f+g$ is bounded from below. Then, the following statements hold true.
\begin{itemize}
\item[(i)] $\omega((y_n)_{n\in\N})=\omega((z_n)_{n\in\N})=\omega((x_n)_{n\in\N})\subseteq\crit(f+g)$, further $\omega((u_n)_{n\in\N})=\{(\ol x,\ol x)\in\R^m\times\R^m: \ol x\in\omega((x_n)_{n\in\N})\}$;
\item[(ii)]  $\crit(H)=\{(x,x)\in\R^m\times\R^m: x\in\crit(f+g)\}$ and $\omega((u_n)_{n\in\N})\subseteq\crit(H)$;
\item[(iii)] $(H(u_n))_{n\in\N}$ is convergent and $H$ is constant on $\omega((u_n)_{n\in\N})$.
\end{itemize}
\end{lemma}
\begin{proof}
(i) Let $\ol x\in\omega((x_n)_{n\in\N}).$ Then, there exists a subsequence $(x_{n_k})_{k\in \N}$ of $(x_n)_{n\in\N}$ such that
$$\lim_{k\to+\infty}x_{n_k}=\ol x.$$
Since by \eqref{forcont} we get that $\lim_{n\To+\infty}(x_n-x_{n-1})=0$ and the sequences $(\a_n)_{n\in\N},\,\left(\b_n\right)_{n\in\N}$ converge, we obtain that
$$\lim_{k\to+\infty}y_{n_k}=\lim_{k\to+\infty}u_{n_k}=\lim_{k\to+\infty}x_{n_k}=\ol x,$$
which shows that $$\omega((x_n)_{n\in\N})\subseteq \omega((z_n)_{n\in\N})\mbox{ and }\omega((x_n)_{n\in\N})\subseteq \omega((y_n)_{n\in\N}).$$

Conversely, if $\ol y\in \omega((y_n)_{n\in\N})$ then, from \eqref{forcont} it results that $\ol y\in \omega((x_n)_{n\in\N}).$ Further, if $\ol z\in \omega((z_n)_{n\in\N})$ then by using \eqref{forcont} again we obtain that  $\ol z\in \omega((y_n)_{n\in\N}).$
Hence,
$$\omega((y_n)_{n\in\N})=\omega((u_n)_{n\in\N})=\omega((x_n)_{n\in\N}).$$

We show next that
$$\omega((x_n)_{n\in\N})\subseteq\crit(f+g).$$

Let $\ol{x}\in\omega((x_n)_{n\in\N})$ and $(x_{n_k})_{k\in \N}$ a subsequence of $(x_n)_{n\in\N}$ such that
$$\lim_{k\To+\infty}x_{n_k}=\ol x.$$
We have to show that $0\in\p(f+g)(\ol{x}).$ From PADISNO or c-PADISNO we have for every $k \geq 1$
$$\frac{1}{s}(y_{n_{k}-1}-x_{n_{k}})-\n g(z_{n_{k}-1})\in\p f(x_{n_{k}})$$
hence,
$$p_k=\frac{1}{s}(y_{n_{k}-1}-x_{n_{k}})+(\n g(x_{n_k})-\n g(z_{n_{k}-1}))\in\p (f+g)(x_{n_{k}}).$$

In virtue of \eqref{forcont}
$$y_{n_{k}-1}-x_{n_{k}}\To 0\mbox{ and }z_{n_{k}-1}-x_{n_{k}}\To 0,\,k\To+\infty,$$
consequently
$$y_{n_{k}-1}\To \ol x\mbox{ and }z_{n_{k}-1}\To \ol x,\,k\To+\infty.$$

Further,
$$\|\n g(x_{n_k})-\n g(z_{n_{k}-1})\|\le L_g\|z_{n_{k}-1}-x_{n_{k}}\|,$$
thus
$$\n g(x_{n_k})-\n g(z_{n_{k}-1})\To 0,\,k\To+\infty.$$

Consequently, $p_k\To 0,\,k\To+\infty.$

We show that $\lim_{k\To+\infty}(f+g)(x_{n_k})=(f+g)(\ol{x}).$ Since $f$ is lower semicontinuous, one has
\begin{equation}\label{finf}
\liminf_{k\To+\infty}f(x_{n_k})\ge f(\ol{x}).
\end{equation}

Further  we have for every $k \geq 1$
\begin{align*}
x_{n_k}&=\argmin_{y\in\R^m}\left[f(y)+\frac{1}{2 s}\|y-y_{n_k-1}+s \nabla g(z_{n_k-1})\|^2\right]  \\
&=\argmin_{y\in\R^m}\left[f(y)+\frac{1}{2s}\|y-y_{n_k-1}\|^2+\<y-y_{n_k-1},\n g(z_{n_k-1})\>+\frac{s}{2}\|\n g(z_{n_k-1})\|^2\right]  \\
&=\argmin_{y\in\R^m}\left[f(y)+\frac{1}{2s}\|y-y_{n_k-1}\|^2+\<y-y_{n_k-1},\n g(z_{n_k-1})\>\right].
\end{align*}

Hence, for every $k \geq 1$ we have
$$f(x_{n_k})+\frac{1}{2s}\|x_{n_k}-y_{n_k-1}\|^2+\<x_{n_k}-y_{n_k-1},\n g(z_{n_k-1})\le f(\ol x)+\frac{1}{2s}\|\ol x-y_{n_k-1}\|^2+\<\ol x-y_{n_k-1},\n g(z_{n_k-1})\>.$$
Taking the limit superior as $k\To+\infty$, we obtain
\begin{equation}\label{fsup}
\limsup_{k\To+\infty}f(x_{n_k})\le f(\ol{x}).
\end{equation}
Now \eqref{finf} and \eqref{fsup} show that $\lim_{k\To+\infty}f(x_{n_k})= f(\ol{x})$ and, since $g$ is continuous,  we obtain
$$\lim_{k\To+\infty}(f+g)(x_{n_k})= (f+g)(\ol{x}).$$
By the  closedness criterion of the graph of the limiting subdifferential it follows that
$$0\in\p(f+g)(\ol{x}).$$
So we have shown that
$$\omega((y_n)_{n\in\N})=\omega((z_n)_{n\in\N})=\omega((x_n)_{n\in\N})\subseteq\crit(f+g).$$

Obviously $\omega((\tilde{x}_n)_{n\in\N})=\omega((x_n)_{n\in\N})$ and since the sequences $(\a_n)_{n\in\N},\,(\b_n)_{n\in\N}$ are bounded, (convergent), from \eqref{forcont} one gets
\begin{equation}\label{rege1}
\lim_{n\To+\infty}\tilde{\d}_n(\tilde{x}_{n}-\tilde{x}_{n-1})=0.
\end{equation}
Let $(\ol x,\ol y)\in \omega((u_n)_{n\in\N}).$ Then, there exists a subsequence $(u_{n_k})_{k\in\N}$ such that
$u_{n_k}\To(\ol x,\ol y),\,k\To+\infty.$
But we have
$u_n=\left(\tilde{x}_{n},\tilde{x}_{n}+\tilde{\d}_n(\tilde{x}_{n}-\tilde{x}_{n-1})\right),$ for all $n\in\N$, consequently from \eqref{rege1} we obtain
$$\tilde{x}_{n_k}\To \ol x\mbox{ and }\tilde{x}_{n_k}\To \ol y,\,k\To+\infty.$$
Hence, $\ol x=\ol y$ and $\ol x\in \omega((x_n)_{n\in\N})$  which shows that
$$\omega((u_n)_{n\in\N})\subseteq\{(\ol x,\ol x)\in\R^m\times\R^m: \ol x\in\omega((x_n)_{n\in\N})\}.$$
Conversely, if $\ol x \in \omega((\tilde{x}_n)_{n\in\N})$ then there exists a subsequence $(\tilde{x}_{n_k})_{k\in \N}$ such that
$\lim_{k\to+\infty}\tilde{x}_{n_k}=\ol x.$ But then, by using \eqref{rege1} we obtain at once that
$u_{n_k}\To (\ol x,\ol x),\,k\To+\infty,$ hence by using the fact that $\omega((\tilde{x}_n)_{n\in\N})=\omega((x_n)_{n\in\N})$ we obtain
$$\{(\ol x,\ol x)\in\R^m\times\R^m: \ol x\in\omega((x_n)_{n\in\N})\}\subseteq \omega((u_n)_{n\in\N}).$$

For (ii) by using the fact that
$$\p H(x,y)=\{\p f(x)+\n g(x)+x-y\}\times\{y-x\}$$
we get
$$\crit H=\{(x,y)\in\R^m\times\R^m:(0,0)\in \{\p f(x)+\n g(x)+x-y\}\times\{y-x\}\}.$$
Hence, $x=y$ and $0\in\p f(x)+\n g(x)=\p (f+g)(x),$
consequently
$$\crit H=\{(x,x)\in\R^m\times\R^m:0\in\p (f+g)(x)\}.$$

Now, since $\omega((u_n)_{n\in\N})\subseteq\{(\ol x,\ol x)\in\R^m\times\R^m: \ol x\in\omega((x_n)_{n\in\N})\}$ and $\omega((x_n)_{n\in\N})\subseteq\crit(f+g)$  and  we have
$$\omega((u_n)_{n\in\N})\subseteq\{(\ol x,\ol x)\in\R^m\times\R^m: \ol x\in \crit(f+g)\}=\crit(H).$$

(iii) Follows directly by (ii) in Lemma \ref{descentnc} or Lemma \ref{descent}.
\end{proof}

Now we are able to prove one of the main result of the paper, namely Theorem \ref{convergence}.

\begin{proof}(\rm Theorem \ref{convergence}) Let $(u_n)_{n\in\N}$ be the sequence defined by \eqref{zform}. Since $x^*\in \omega((x_n)_{n\in\N})$ according to Lemma \ref{regularization} (i) one has $x^*\in\crit(f+g)$ and  $u^*=(x^*,x^*)\in\omega((u_n)_{n\in\N}).$

From Proposition \ref{condH} we get that the assumptions (H1)-(H3) of Theorem \ref{thabstrconv} are satisfied with  the  function $H$, the sequences $(u_n)_{n\in\N}$ and $(\tilde{x}_n)_{n\in\N}.$

It remained to show (H4). We have shown in the proof of Lemma \ref{regularization} that if $x^*\in\omega((x_n)_{n\in\N})$ and
$x_{n_k}\To x^*,\,k\To+\infty$, then
$$(f+g)(x_{n_k})\To (f+g)(x^*),\,k\To+\infty.$$
But then, by using \eqref{forcont} we get that
$$u_{n_k}\To (x^*,x^*)=u^*,\mbox{ and }H(u_{n_k})\To (f+g)(x^*)=H(u^*),\,k\To+\infty.$$

 Hence, according to Theorem \ref{thabstrconv}, the sequence $(\tilde{x}_n)_{n\in\N}$ converges to $x^*$ as $n\To+\infty$. But then obviously the sequence $({x}_n)_{n\in\N}$ converges to $x^*$ as $n\To+\infty$.
\end{proof}

\begin{remark}\label{rateyz} Note that under the assumptions of Theorem \ref{convergence} we also have that
$$\lim_{n\To+\infty}y_n=\lim_{n\To+\infty}z_n= x^*\mbox{ and }$$
$$\lim_{n\To+\infty}(f+g)(x_n)=\lim_{n\To+\infty}(f+g)(y_n)=\lim_{n\To+\infty}(f+g)(z_n)=(f+g)(x^*).$$
\end{remark}

\begin{corollary}\label{sa} In the settings of problem \eqref{propt}, for some starting points $x_{-1}=x_0\in\R^m,$ consider the sequence $(x_n)_{n\in\N}$ generated by PADISNO or c-PADISNO. Assume that $f+g$ is semi-algebraic and bounded from below. Assume further that $\omega((x_n)_{n\in\N})\neq\emptyset.$
%Then, the sequence $(x_n)_{n\in\N}$ converges to $x^*$, $(z_n)_{n\in\N}$ converges to $z^*$ and $z^*\in\crit(F).$
 Then, the sequence $(x_n)_{n\in\N}$ converges to a critical point of the objective function $f+g.$
\end{corollary}
\begin{proof}
  Since the class of semi-algebraic functions is closed under addition (see for example \cite{b-sab-teb}) and
$(x,y) \mapsto \frac12\|x-y\|^2$ is semi-algebraic, we obtain that the  function
$$H:\R^m\times\R^m\To\R\cup\{+\infty\},\,H(x,y)=(f+g)(x)+\frac12\|y-x\|^2$$ is semi-algebraic. Consequently, $H$ is a KL function. In particular $H$ has the Kurdyka-{\L}ojasiewicz property at a point $z^*=(x^*,x^*),$ where $x^*\in \omega((x_n)_{n\in\N}).$ The conclusion follows from Theorem \ref{convergence}.
\end{proof}

\begin{remark}\label{r2} In order to apply Theorem \ref{convergence} or Corollary \ref{sa} we need to assume that $\omega((x_n)_{n\in\N})$ is nonempty. Obviously, this condition is satisfied whenever the sequence $(x_n)_{n\in\N}$ is bounded. Note that the boundedness of $(x_n)_{n\in\N}$ is guaranteed if we assume that the objective function $f+g$ is coercive, that is,
$\lim_{\|x\|\To+\infty}(f+ g)(x)=+\infty.$
\end{remark}

An immediate consequence of Theorem \ref{convergence} and Remark \ref{r2} is the following result.

\begin{corollary}\label{fornumex} Assume that $f+g$ is a coercive function. In the settings of problem \eqref{propt}, for some starting points $x_{-1}=x_0\in\R^m,$ consider the sequence $(x_n)_{n\in\N}$ generated by PADISNO or c-PADISNO.  Assume further that
$$H:\R^m\times\R^m\To\R\cup\{+\infty\},\,H(x,y)=(f+g)(x)+\frac12\|y-x\|^2$$ is a KL function.

 Then, the sequence $(x_n)_{n\in\N}$ converges to a critical point of the objective function $f+g.$
\end{corollary}

\section{Convergence rates for the numerical schemes PADISNO and c-PADISNO}

In this section we  prove Theorem \ref{convergencerates} concerning the convergence rates for the sequences generated by PADISNO and c-PADISNO in terms of the KL exponent of the regularization function $H.$ Further, some particular instances of Theorem \ref{convergencerates} will be discussed. %obtain convergence rates for PADISNO and c-PADISNO in terms of the KL exponent of the regularization function $H.$ We have the following result.

\begin{proof}(\rm Theorem \ref{convergencerates}) The fact that the sequences $(x_n)_{n\in\N},\,(y_n)_{n\in\N}$ and $(z_n)_{n\in\N}$ converge to $x^*$ and $x^*$ is a critical point of the objective function $f+g$ follows directly from Theorem \ref{convergence} and Remark \ref{rateyz}. In order to prove (a)-(c) we apply Theorem \ref{thabstrrate} to the function $H$ and the sequences $(u_n)_{n\in\N}$ and $(\tilde{x}_{n})_{n\in\N}$ defined by \eqref{zform}.

 {\bf (a)} Assume that $\t=0$. Taking into account that
$$H(u_n)=(f+g)(\tilde{x}_n)+\d_{n+N}\|\tilde{x}_n-\tilde{x}_{n-1}\|^2,\mbox{ for all }n\in\N$$
according to Theorem \ref{thabstrrate} we have  that $((f+g)(\tilde{x}_n)+\d_{n+N}\|\tilde{x}_n-\tilde{x}_{n-1}\|^2)_{n\in\N},$ $(\tilde{x}_n)_{n\in\N}$ and $(u_n)_{n\in\N}$ converge in a finite number of steps. But then $\tilde{x}_n=x^*$  after an index $N_0\in \N$. This leads to $x_n=x^*$ for all $n\ge N+N_0$. Consequently, $(f+g)(x_n)=(f+g)(x^*)$  for all $n\ge N+N_0$. The forms of  the sequences $(y_n)_{n\in\N},\,(z_n)_{n\in\N}$ lead at once that
$$y_n=x^*,\,z_n=x^*,\mbox{ for all }n\ge N+N_0+1.$$

Further,   $\tilde{x}_n-\tilde{x}_{n-1}=0$ for $n\ge N_0+1$, hence by using the fact that $u_n=\left(\tilde{x}_{n},\tilde{x}_{n}+\tilde{\d}_n(\tilde{x}_{n}-\tilde{x}_{n-1})\right)$ we get that $$u_n=(x^*,x^*),\mbox{ for all }n\ge N+N_0+1.$$
\vskip0.3cm

{\bf (b)} Assume that $\t\in\left(0,\frac12\right].$ Then, according to Theorem \ref{thabstrrate}, there exist $A_1>0,$ $Q\in[0,1)$  and $\ol k\in\N$ such that
$H(u_{n})-H(u^* )\le A_1 Q^n$, $\|\tilde{x}_n-x^*\|\le A_1 Q^{\frac{n}{2}}$ and $\|u_n-u^*\|\le A_1 Q^{\frac{n}{2}}$ for  every $n\ge \ol k$.

Hence,
$$(f+g)(\tilde{x}_n)+\d_{n+N}\|\tilde{x}_n-\tilde{x}_{n-1}\|^2-(f+g)(x^*)\le A_1 Q^n,\mbox{ for all }n\ge\ol k$$
which leads to
\begin{equation}\label{forratefg}
(f+g)({x}_n)-(f+g)(x^*)\le \left(A_1 Q^{-N}\right) Q^n,\mbox{ for all }n\ge\ol k+N.
\end{equation}

Further, $\|\tilde{x}_n-x^*\|\le A_1 Q^{\frac{n}{2}}$  for all $n\ge\ol k$ yields
\begin{equation}\label{forratex}
\|x_n-x^*\|\le \left(A_1 Q^{-\frac{N}{2}}\right) Q^{\frac{n}{2}},\mbox{ for all }n\ge\ol k+N.
\end{equation}

Now, obviously
$$\|x_n-x_{n-1}\|\le\|x_n-x^*\|+\|x_{n-1}-x^*\|\mbox{ for all }n\ge 1,$$ hence
$$\|x_n-x_{n-1}\|\le \left(A_1 Q^{-\frac{N}{2}}\right)\left(1+Q^{-\frac12}\right) Q^{\frac{n}{2}},\mbox{ for all }n\ge\ol k+N+1.$$
From the latter relation and the fact that
$\|y_n-x^*\|\le\|x_n-x^*\|+|\a_n|\|x_n-x_{n-1}\|$ and $|\a_n|$ is bounded, we obtain that there exists $M_1>0$ such that
\begin{equation}\label{forratey}
\|y_n-x^*\|\le M_1 Q^{\frac{n}{2}},\mbox{ for all }n\ge\ol k+N+1.
\end{equation}

By similar arguments we obtain that there exists $M_2>0$  such that
\begin{equation}\label{forratez}
\|z_n-x^*\|\le M_2 Q^{\frac{n}{2}},\mbox{ for all }n\ge\ol k+N+1.
\end{equation}

So,  if we take $\ol N=\ol k+N+1$ and $a_1=\max\left(A_1 Q^{-N},M_1,M_2\right)$ then \eqref{forratefg}, \eqref{forratex}, \eqref{forratey} and \eqref{forratez} show (b).
\vskip0.3cm

{\bf (c)} Assume that $\t\in\left(\frac12,1\right).$ Then, according to Theorem \ref{thabstrrate} there exist  $A_2>0$ and $\ol k\in\N$ such that
$H(u_n)-H(u^*)\le A_2 {n^{-\frac{1}{2\t-1}}},$ $\|\tilde{x}_n-x^*\|\le A_2 {n^{\frac{\t-1}{2\t-1}}}$  and $\|u_n-u^*\|\le A_2 {n^{\frac{\t-1}{2\t-1}}}\mbox{ for all }n\ge\ol k$.

Hence,
$$(f+g)(\tilde{x}_n)+\d_{n+N}\|\tilde{x}_n-\tilde{x}_{n-1}\|^2-(f+g)(x^*)\le A_2 {n^{-\frac{1}{2\t-1}}},\mbox{ for all }n\ge\ol k$$
which leads to
$$
(f+g)({x}_n)-(f+g)(x^*)\le A_2 {(n-N)^{-\frac{1}{2\t-1}}},\mbox{ for all }n\ge\ol k+N.$$
Now, $\sup_{n\ge\ol k+N}A_2 \left(\frac{n-N}{n}\right)^{-\frac{1}{2\t-1}}\le M_3<+\infty,$ hence
\begin{equation}\label{forratefg1}
(f+g)({x}_n)-(f+g)(x^*)\le M_3{n^{-\frac{1}{2\t-1}}},\mbox{ for all }n\ge\ol k+N.
\end{equation}
Since $\|\tilde{x}_n-x^*\|\le A_2 {n^{\frac{\t-1}{2\t-1}}}$ for all $n\ge\ol k$ we get
\begin{equation}\label{forratex1}
\|x_n-x^*\|\le A_2 {(n-N)^{\frac{\t-1}{2\t-1}}}=A_2 \left(\frac{n-N}{n}\right)^{\frac{\t-1}{2\t-1}}n^{\frac{\t-1}{2\t-1}}\le M_4 n^{\frac{\t-1}{2\t-1}},
\end{equation}
 for all $n\ge\ol k+N$, where $M_4=\sup_{n\ge\ol k+N}A_2 \left(\frac{n-N}{n}\right)^{\frac{\t-1}{2\t-1}}<+\infty.$
Further,
$$\|x_n-x_{n-1}\|\le\|x_n-x^*\|+\|x_{n-1}-x^*\|\le  M_4 n^{\frac{\t-1}{2\t-1}}+ M_4 (n-1)^{\frac{\t-1}{2\t-1}},\mbox{ for all }n\ge\ol k+N+1$$
and $\sup_{n\ge \ol k+N+1}\left(\frac{n-1}{n}\right)^{\frac{\t-1}{2\t-1}}<+\infty$, hence there exists $M_5>0$ such that
$$|x_n-x_{n-1}\|\le M_5 n^{\frac{\t-1}{2\t-1}},\mbox{ for all }n\ge\ol k+N+1.$$

From the latter relation and the facts that
$\|y_n-x^*\|\le\|x_n-x^*\|+|\a_n|\|x_n-x_{n-1}\|$ and $|\a_n|$ is bounded, further  $\|z_n-x^*\|\le\|z_n-x^*\|+|\b_n|\|x_n-x_{n-1}\|$ and $|\b_n|$ is bounded, we obtain that there exists $M_6>0$ such that
\begin{equation}\label{forratey1}
\|y_n-x^*\|\le M_6 n^{\frac{\t-1}{2\t-1}},\mbox{ for all }n\ge\ol k+N+1
\end{equation}
and
\begin{equation}\label{forratez1}
\|z_n-x^*\|\le M_6 n^{\frac{\t-1}{2\t-1}} ,\mbox{ for all }n\ge\ol k+N+1.
\end{equation}

So,  if we take $\ol N=\ol k+N+1$ and $a_2=\max\left(M_3,M_4,M_5,M_6\right)$ then from \eqref{forratefg1}, \eqref{forratex1}, \eqref{forratey1} and \eqref{forratez1} we obtain (c).
\end{proof}

According to Theorem 3.6 \cite{LP}, if $f+g$ has the KL property with KL exponent $\t\in\left[\frac12,1\right)$ at $\ol x\in\R^m,$ then the function $H:\R^m\times\R^m\To\R\cup\{+\infty\},\,H(x,y)=(f+g)(x)+\frac12\|y-x\|^2$ has the KL property at $(\ol x,\ol x)\in\R^m\times\R^m$ with the same KL exponent $\t.$ This result allows us to reformulate Theorem \ref{convergencerates}.

\begin{corollary} In the settings of problem \eqref{propt}, for some starting points $x_{-1}=x_0\in\R^m,$  consider the sequence $(x_n)_{n\in\N}$ generated by PADISNO or c-PADISNO. Assume that $f+g$ is bounded from below and  has the Kurdyka-{\L}ojasiewicz property at $x^*\in\omega((x_n)_{n\in\N})$, (which obviously must be assumed  nonempty), with KL exponent $\t\in\left[\frac12,1\right)$. If $\t=\frac12$ then the convergence rates stated at {\rm Theorem \ref{convergencerates}(b)},  if  $\t\in\left(\frac12,1\right)$  then the convergence rates stated at  {\rm Theorem \ref{convergencerates}(c)} hold.
\end{corollary}
\begin{proof}
 Indeed, from Lemma \ref{regularization} (i) one has $z^*=(x^*,x^*)\in\omega((z_n)_{n\in\N})$ and according to Theorem 3.6 \cite{LP}
$H$ has the Kurdyka-{\L}ojasiewicz property at $z^*$ with KL exponent $\t.$ Hence, Theorem \ref{convergencerates} can be applied.
\end{proof}

In case we assume that the function $f+g$ is strongly convex, then Theorem \ref{convergencerates} assures linear convergence rates for the sequences generated by PADISNO or c-PADISNO. The following result holds.

\begin{theorem}\label{ratestrconv}
In the settings of problem \eqref{propt}, for some starting points $x_{-1}=x_0\in\R^m,$  consider the sequences $(x_n)_{n\in\N},\,(y_n)_{n\in\N}$ and $(z_n)_{n\in\N}$ generated by PADISNO or c-PADISNO. Assume that the objective function $f+g$ is strongly convex and let $x^*$ be the unique minimizer of $f+g.$ Then, there exist $Q\in[0,1)$, $a_1>0$ and $\ol N\in\N$ such that the following statements hold true:
\begin{itemize}
\item[(i)] $(f+g)(x_{n})-(f+g)(x^* )\le a_1 {Q^n}$ for  every $n\ge \ol N$,
\item[(ii)] $\|x_n-x^* \|\le a_1 {Q^{\frac{n}{2}}}$, $\|y_n-x^* \|\le a_1 {Q^{\frac{n}{2}}}$ and $\|z_n-x^* \|\le a_1 {Q^{\frac{n}{2}}}$ for  every $n\ge\ol N$.
\end{itemize}
\end{theorem}
\begin{proof}
We emphasize that the strongly convex function $f+g$  is coercive, see \cite{BauComb}. According to \cite{rock-wets} the function $f+g$ is bounded from bellow. According to Lemma \ref{regularization} (i) and the hypotheses of the theorem, $\omega((x_n)_{n\in\N})=\{x^*\},$ hence $x_n\To x^*,\,n\To+\infty.$ According to \cite{attouch-bolte2009}, $f+g$ satisfies the Kurdyka-{\L}ojasiewicz property at $x^*$ with the KL exponent $\t=\frac12.$ Then, according to Theorem 3.6 \cite{LP}, $H$ satisfies the Kurdyka-{\L}ojasiewicz property at $(x^*,x^*)$ with the same KL exponent $\t=\frac12.$ The conclusion now follows from Theorem \ref{convergencerates}.
\end{proof}

\section{Applications to image restoration}

In what follows we apply PADISNO to image restoration. For this purpose we write   the image restoration problem as an optimization problem having in its objective the sum of a non-convex misfit functional and a non-convex non-smooth regularization, (see also \cite{BCL}).

For a given blur operator $A \in \mathbb{R}^{m \times m}$  and a given blurred and noisy image $b \in \R^m$, the image restoration problem consists in estimating the unknown original image $\ol x\in\R^m$ fulfilling $A\ol x=b.$

To this end we solve the   non-convex minimization problem \eqref{propt}
with
$$f(x)=\lambda \|Wx\|_0,\,\l>0,\,g(x)=\sum_{k=1}^M\sum_{l=1}^N\log\left(1+(Ax-b)_{kl}^2\right).$$
 Here $\lambda >0$ is a regularization parameter,
$W:\R^m\rightarrow\R^m$ is a discrete Haar wavelet transform with
four levels and $\|y\|_0=\sum_{i=1}^m|y_i|_0$ ($|\cdot|_0 = |\sgn(\cdot)|$) furnishes the number of nonzero entries of the vector $y=(y_1,...,y_m)\in\R^m$. In this context, $x \in \R^m$
represents the vectorized image $X\in\R^{M\times N}$, where $m = M\cdot N$ and $x_{i,j}$
denotes the normalized value of the pixel located in the $i$-th row and the $j$-th column, for
$i=1,\ldots,M$ and $j=1,\ldots,N$. Arguing as in \cite{BCL}, one can conclude that $H$ in Theorem \ref{convergence} is a KL function.

According to \cite{BCL} the proximal operator of $f$, (which in non-convex case it is not single valued anymore), is
$$\prox\nolimits_{\gamma f}(x)=W^*\prox\nolimits_{\lambda\gamma\|\cdot\|_0}(Wx) \ \forall x \in \R^m \ \forall \gamma >0,$$
where for all $u=(u_1,...,u_m)$ we have
$$\prox\nolimits_{\lambda\gamma\|\cdot\|_0}(u)=(\prox\nolimits_{\lambda\gamma|\cdot|_0}(u_1),...,\prox\nolimits_{\lambda\gamma|\cdot|_0}(u_m))$$
and for all $t\in\R$ $$\prox\nolimits_{\lambda\gamma|\cdot|_0}(t)=\left\{
\begin{array}{ll}
t, & \mbox {if } |t|>\sqrt{2\lambda\gamma},\\
\{0,t\}, & \mbox {if } |t|=\sqrt{2\lambda\gamma},\\
0, & \mbox {otherwise.}
\end{array}\right.$$
In our first experiments we used the $256 \times 256$ cameraman test image which we first blurred by using a Gaussian
blur operator of size $9 \times 9$ and standard deviation $4$ and to which we afterward added a zero-mean white
Gaussian noise with standard deviation $10^{-3}$. We took as regularization parameter $\lambda=10^{-5}$ and in PADISNO we considered different constant inertial parameters $\alpha_n=\a,\,\b_n=\b$ for all $n \geq 1$. Then, the corresponding step size is taken as $s=0.999\cdot\frac{1-2|\a|}{L(2|\b|+1)}$, where the Lipschitz constant of the gradient of the smooth misfit function $g$ is
$L_{g} = 2$. We run PADISNO for 300 iterates. The results obtained, depicted at Figure 5, show the indeed by allowing different and also negative inertial parameters in PADISNO one may expect a better performance.

\begin{figure}[hbt!]
%\begin{subfigure}{.99\textwidth}
  \centering
  \includegraphics[width=.99\linewidth]{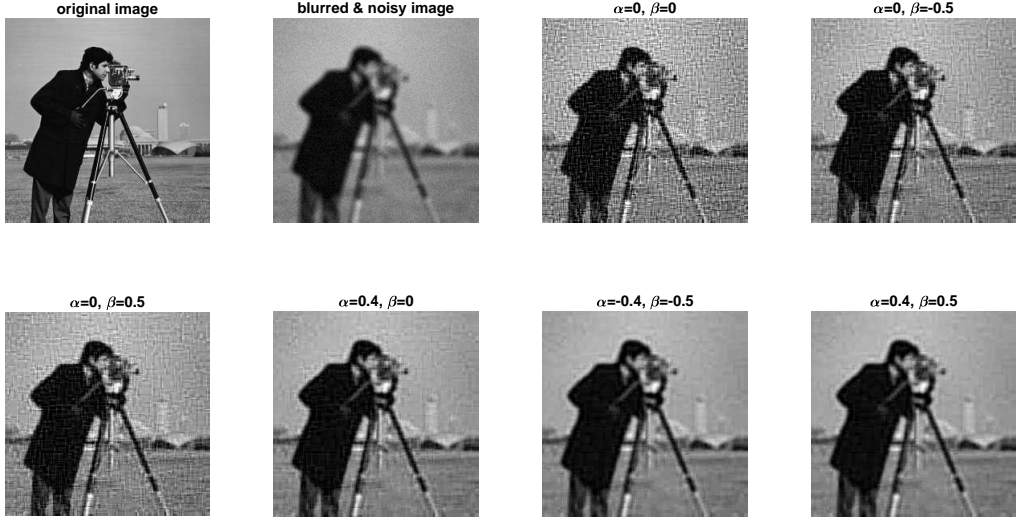}
 % \caption{$x_{-1}=x_0=(0.5,-0.5).$}
  %\label{fig2:sfig21}
%\end{subfigure}
 \caption{\small Reconstruction  of a blurred and noisy image with PADISNO}
\end{figure}

In our second experiment we used the same $256 \times 256$ cameraman test image for which we add a salt and pepper noise, with $0.3$ noise density. The other characteristics are the same as in our previous experiment. The results obtained are depicted at Figure 6.

\begin{figure}[hbt!]
%\begin{subfigure}{.99\textwidth}
  \centering
  \includegraphics[width=.99\linewidth]{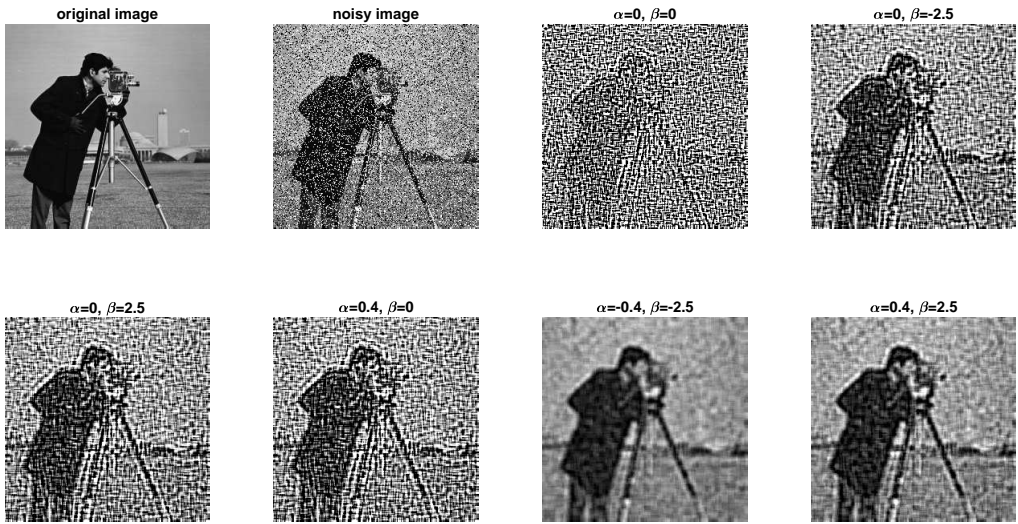}
 % \caption{$x_{-1}=x_0=(0.5,-0.5).$}
  %\label{fig2:sfig21}
%\end{subfigure}
 \caption{\small Reconstruction of a noisy image with PADISNO}
\end{figure}

Observe that the best result is obtained for $\a=-0.4$ and $\b=-2.5.$ Further, in this experiment we also compared the quality of the recovered images for different values of $\a$ and $\b$ by making use of the improvement in signal-to-noise ratio
(ISNR), which is defined as
$$ \text{ISNR}(n) = 10 \log_{10}\left( \frac{\left\|x-b\right\|^2}{\left\|x-x_n\right\|^2} \right),$$
where $x$, $b$ and $x_n$ denote the original, observed and estimated image at iteration $n$, respectively.

In Table \ref{tab} we list the values of the ISNR-function after $300$ iterations.
One can notice that for negative and different inertial parameters $\a$ and $\b$ we obtain considerable better results.

\begin{table}[hbt!]
\begin{tabular}{ccccccccccccccc}
\hline
$\a$ &       $0$     & $0$     & $0$     &  $-0.4$ & $-0.4$ & $-0.4$ & $-0.4$  & $-0.4$  & $0.4$  & $0.4$   & $0.4$   & $0.4$ & $0.4$ \\
\hline
$\b$ &       $0$     & $-2.5$  & $2.5$   & $-2.5$  & $-0.4$ & $0$    & $0.4$  & $2.5$    & $-2.5$ & $-0.4$  & $0$     & $0.4$ & $2.5$\\
\hline
ISNR(300) & $-12.58$ & $-1.70$ & $-1.70$ & $8.88$ & $4.70$ & $1.45$ & $4.70$  & $8.88$   & $6.96$ & $0.81$  & $-2.86$ & $0.81$ & $6.96$\\
\hline
\end{tabular}
\caption{The ISNR values after 300 iterations for different choices of $\a$ and $\b$.}
\label{tab}
\end{table}

\section{Conclusions}

The novelty of the two forward-backward inertial algorithms studied in the present paper in connection of a structured non-convex optimization problem consist in allowing in the numerical schemes different and also negative inertial parameters. This way we get a better control on the step size and, as some numerical experiments show, our algorithms will have a superior behaviour compared to the known algorithms from the literature where the inertial parameters are equal, (or $\b_n\equiv 0$), and non-negative. The convergence of a sequence generated by our algorithms is obtained by deploying the KL property of a regularization of the objective function. Further, the well known convergence rates are obtained in terms of the KL exponent of this regularization.

The two forward-backward inertial algorithms studied in the present paper in connection to the minimization of the sum of a non-smooth function and a smooth function offer several possibilities for future researches. A first research line is to renounce to the assumption that the gradient of $g$ is Lipschitz continuous and instead of constant step size use an adaptive step size. This can be done in a natural way, by replacing at every iterative step $L_g,$ in the formula that gives the upper bound  for the step size in PADISNO or c-PADISNO, with the local Lipschitz constant of $\n g$ on the segment $[x_{n-1},x_n].$

Another interesting topic is to apply the algorithms studied in this paper to DC programming, by assuming that in the optimization problem \eqref{propt} the function $f$ is convex and $g$ is concave. As we emphasized before, in this case for an appropriate choice of the inertial parameters $(\a_n)$ and $(\b_n)$ the stepsize $s$ can be taken arbitrary large.

\appendix
%\section*{Appendix}
%\section{Auxiliary results}
\section{Proofs of Abstract Convergence Results}

In what follows  we give  full proofs for Lemma \ref{abstrconv}, Corollary \ref{corabstrconv}, Theorem \ref{thabstrconv} and Theorem \ref{thabstrrate}. %We have the following result.

 \begin{proof}(\rm Proof of Lemma \ref{abstrconv})
We divide the proof into the following steps.

{\bf Step I.} We show that $u_1\in B(u^*,\rho)$  and $F(u_1)< F(u^*)+\eta.$

Indeed, $u_0\in B(u^*,\rho)$ and \eqref{e00} assures that $F(u_1)\ge F(u^*).$
 Further, (H1) assures that $$\|x_1-x_0\|\le\sqrt{\frac{F(u_0)-F(u_1)}{a}}\le\sqrt{\frac{F(u_0)-F(u^*)}{a}}.$$
  Since $\|x_1-x^*\|=\|(x_1-x_0)+(x_0-x^*)\|\le \|x_1-x_0\|+\|x_0-x^*\|$  and
 $F(u_1)\le F(u_0)$  the condition  \eqref{e01} leads to
  $$\|x_1-x^*\|\le \|x_0-x^*\|+ \sqrt{\frac{F(u_0)-F(u^*)}{a}}<\frac{\rho}{c_1+c_2}.$$
 Now, from (H3) we have $\|u_1-u^*\|\le c_1\|x_1-x^*\|+c_2\|x_{0}-x^*\|$ hence
  $$\|u_1-u^*\|< c_1\frac{\rho}{c_1+c_2}+c_2\frac{\rho}{c_1+c_2}= \rho.$$
  Thus, $u_1\in B(u^*,\rho),$ moreover \eqref{e00} and (H1) provide that $F(u^*)\le F(u_2)\le F(u_1)\le F(u_0)< F(u^*)+\eta.$

{\bf Step II.} Next we show that whenever for a $k\ge 1$ one has $u_k\in B(u^*,\rho),\,F(u_k)<F(u^*)+\eta$ then it holds that
\begin{equation}\label{fontos}
3\|x_{k+1}-x_{k}\|\le\|x_{k}-x_{k-1}\|+\|x_{k-1}-x_{k-2}\|+\frac{9b}{4a}(\varphi(F(u_k)-F(u^*))-\varphi(F(u_{k+1})-F(u^*))).
\end{equation}
Hence, let $k\ge 1$ and assume that  $u_k\in B(u^*,\rho),\,F(u_k)<F(u^*)+\eta$. Note that from (H1) and \eqref{e00} one has $F(u^*)\le F(u_{k+1})\le F(u_k)<F(u^*)+\eta,$ hence $$F(u_k)-F(u^*),F(u_{k+1})-F(u^*)\in[0,\eta),$$ thus \eqref{fontos} is well stated. Now, if $x_k=x_{k+1}$ then \eqref{fontos} trivially holds.

Otherwise, from (H1) and \eqref{e00} one has
\begin{equation}\label{inter1}
F(u^*)\le F(u_{k+1})<F(u_k)<F(u^*)+\eta.
\end{equation}
Consequently,  $u_k\in B(u^*,\rho)\cap  \{u\in\R^m: F(u^*)<F(u)<F(u^*)+\eta\}$ and $B(u^*,\rho)\subseteq B(u^*,\sigma)\subseteq U$ hence, by using the KL inequality  we get
$$\varphi'(F(u_k)-F(u^*))\dist(0, \p F(u_k))\geq 1.$$
Since $\varphi$ is concave, and  \eqref{inter1} assures that $F(u_{k+1})-F(u^*)\in[0,\eta),$ one has
$$\varphi(F(u_k)-F(u^*))-\varphi(F(u_{k+1})-F(u^*))\ge\varphi'(F(u_k)-F(u^*))(F(u_k)-F(u_{k+1})),$$
consequently,
$$\varphi(F(u_k)-F(u^*))-\varphi(F(u_{k+1})-F(u^*))\ge\frac{F(u_k)-F(u_{k+1})}{\dist(0, \p F(u_k))}.$$
Now, by  using (H1) and (H2) we get  that
$$\varphi(F(u_k)-F(u^*))-\varphi(F(u_{k+1})-F(u^*))\ge\frac{a\|x_{k+1}-x_{k}\|^2}{b(\|x_{k}-x_{k-1}\|+\|x_{k-1}-x_{k-2}\|)}.$$
Consequently,
$$\|x_{k+1}-x_{k}\|\le\sqrt{\frac{b}{a}\left(\varphi(F(u_k)-F(u^*))-\varphi(F(u_{k+1})-F(u^*))\right)(\|x_{k}-x_{k-1}\|+\|x_{k-1}-x_{k-2}\|)}$$
and by arithmetical-geometrical mean inequality we have
\begin{align}\nonumber\|x_{k+1}-x_{k}\|&\le\frac{\|x_{k}-x_{k-1}\|+\|x_{k-1}-x_{k-2}\|}{3}\\
\nonumber&+\frac{3b}{4a}(\varphi(F(u_k)-F(u^*))-\varphi(F(u_{k+1})-F(u^*))),
\end{align}
which leads to \eqref{fontos}.%, that is
%$$3\|x_{k+1}-x_{k}\|\le\|x_{k}-x_{k-1}\|+\|x_{k-1}-x_{k-2}\|+\frac{9b}{4a}(\varphi(F(u_k)-F(u^*))-\varphi(F(u_{k+1})-F(u^*))).$$

{\bf Step III.} Now we show by induction that \eqref{fontos} holds for every $k\ge 1.$  Indeed, Step II. can be applied for $k=1$ since according to Step I. $u_1\in B(u^*,\rho)$  and $F(u_1)< F(u^*)+\eta.$ Consequently, for $k=1$ the inequality \eqref{fontos} holds.

 Assume that \eqref{fontos} holds for every $k\in\{1,2,...,n\}$ and we show also that \eqref{fontos} holds for $k=n+1.$ Arguing as at Step II., the condition (H1) and \eqref{e00} assure that $F(u^*)\le F(u_{n+1})\le F(u_n)<F(u^*)+\eta,$ hence it remains to show that $u_{n+1}\in B(u^*,\rho).$
 By using the triangle inequality and (H3) one has
 \begin{align}\label{fonti}
 \|u_{n+1}-u^*\|&\le c_1\|x_{n+1}-x^*\|+c_2\|x_n-x^*\|\\
 \nonumber&=c_1\|(x_{n+1}-x_n)+(x_n-x_{n-1})+\cdots+(x_0-x^*)\|\\
 \nonumber&+c_2\|(x_{n}-x_{n-1})+(x_{n-1}-x_{n-2})+\cdots+(x_0-x^*)\|\\
\nonumber&\le c_1\|x_{n+1}-x_n\|+(c_1+c_2)\|x_0-x^*\|+(c_1+c_2)\sum_{k=1}^{n}\|x_{k}-x_{k-1}\|.
\end{align}

 By summing up \eqref{fontos} from $k=1$ to $k=n$ and using $x_{-1}=x_0$ we obtain
 \begin{equation}\label{fontos1}
\sum_{k=1}^n \|x_k-x_{k-1}\|\le 3\|x_{1}-x_{0}\|-3\|x_{n+1}-x_{n}\|-\|x_n-x_{n-1}\|+\frac{9b}{4a}(\varphi(F(u_1)-F(u^*))-\varphi(F(u_{n+1})-F(u^*))).
\end{equation}

Combining \eqref{fonti} and \eqref{fontos1} and neglecting the negative terms we get
$$\|u_{n+1}-u^*\|\le (3c_1+3c_2)\|x_{1}-x_0\|+(c_1+c_2)\|x_0-x^*\|+(c_1+c_2)\frac{9b}{4a}\varphi(F(u_1)-F(u^*)).$$

But $\varphi$ is strictly increasing and $F(u_1)-F(u^*)\le F(u_0)-F(u^*)$, hence
$$\|u_{n+1}-u^*\|\le (3c_1+3c_2)\|x_{1}-x_0\|+(c_1+c_2)\|x_0-x^*\|+(c_1+c_2)\frac{9b}{4a}\varphi(F(u_0)-F(u^*)).$$

According to (H1) one has
\begin{equation}\label{x1}
\|x_{1}-x_0\|\le\sqrt{\frac{F(u_{0})-F(u_1)}{a}}\le \sqrt{\frac{F(u_{0})-F(u^*)}{a}},
\end{equation}
hence, from \eqref{e01} we get
$$\|u_{n+1}-u^*\|\le (c_1+c_2)\left(\|x_0-x^*\|+3\sqrt{\frac{F(u_{0})-F(u^*)}{a}}+\frac{9b}{4a}\varphi(F(u_0)-F(u^*))\right)<\rho.$$

Hence, we have shown so far that $u_n\in B(u^*,\rho)$ for all $n\in\N.$

{\bf Step IV.} According to Step III. the relation \eqref{fontos} holds for every $k\ge 1.$ But this implies that \eqref{fontos1} holds for every $n\ge 1.$ By using \eqref{x1}  and neglecting the nonpositive terms, \eqref{fontos1} becomes
\begin{equation}\label{fontos2}
\sum_{k=1}^n \|x_k-x_{k-1}\|\le 3\sqrt{\frac{F(u_{0})-F(u^*)}{a}}+\frac{9b}{4a}\varphi(F(u_1)-F(u^*)).
\end{equation}
Now letting $n\To+\infty$ in \eqref{fontos2} we obtain that
$$\sum_{k=1}^{\infty} \|x_k-x_{k-1}\|<+\infty.$$

Obviously the sequence $S_n=\sum_{k=1}^n\|x_k-x_{k-1}\|$ is Cauchy, hence, for all $\e>0$ there exists $N_{\e}\in \N$   such that for all $n\ge N_\e$ and for all $p\in \N$ one has
$$S_{n+p}-S_n\le \e.$$
But
$$S_{n+p}-S_n=\sum_{k={n+1}}^{n+p}\|x_k-x_{k-1}\|\ge \left\|\sum_{k={n+1}}^{n+p}(x_k-x_{k-1})\right\|=\|x_{n+p}-x_n\|$$
hence the sequence $(x_n)_{n\in\N}$ is Cauchy, consequently is convergent. Let
$$\lim_{n\To+\infty}x_n=\ol x.$$

Let $\ol u=(\ol x,\ol x).$ Now, from (H3) we have
$$\lim_{n\To+\infty}\|u_n-\ol u\|\le \lim_{n\To+\infty}(c_1\|x_n-\ol x\|+c_2\|x_{n-1}-\ol x\|)=0,$$
consequently $(u_n)_{n\in\N}$ converges to $\ol u.$

Further, $(u_n)_{n\in\N}\subseteq B(u^*,\rho)$ and $\rho<\sigma$, hence $\ol u\in B(u^*,\sigma).$

 Since $F(u^*)\le F(u_n)<F(u^*)+\eta$ for all $n\ge 1$ and the sequence $(F(u_n))_{n\ge 1}$ is decreasing, obviously $F(u^*)\le \lim_{n\To+\infty} F(u_n)<F(u^*)+\eta.$ Assume that $F(u^*)< \lim_{n\To+\infty} F(u_n).$ Then, one has
$$ u_n\in B(u^*,\sigma)\cap  \{z\in\R^m: F(u^*)<F(z)<F(u^*)+\eta\}$$ and by using the KL inequality and the fact that $\varphi$ is concave, therefore $\varphi'$ is decreasing, we get
$$\varphi'(\lim_{n\To+\infty}F(u_n)-F(u^*))\|\dist(0,\p F(u_n))\|\ge\varphi'(F(u_n)-F(u^*))\|\dist(0,\p F(u_n))\|\geq 1,$$ for all $n\ge 1,$
impossible, since according to (H2) and the fact that $(x_n)_{n\in\N}$ converges one has $$\lim_{n\To+\infty}\|\dist(0,\p F(u_n))\|=0.$$

Consequently, one has $\lim_{n\To+\infty}F(u_n)=F(u^*).$ Since $u_n\To \ol u,\, n\To+\infty$ an $F$ is lower semi-continuous it is obvious that $\lim_{n\To+\infty} F(u_n)\ge F(\ol u).$ Hence,
$$\lim_{n\To+\infty}F(u_n)=F(u^*)\ge F(\ol u).$$

Assume now that (H4) also holds. Obviously in this case
$$u_{n_j}\To\ol u\mbox{ and }F(u_{n_j})\To F(\ol u),\,j\To+\infty.$$
Consequently, one has $F(\ol u)=F(u^*).$

From (H2) we have that there exists $W_{n_j}\in \p F(u_{n_j})$ such that
$$ \|W_{n_j}\|\le b(\|x_{{n_j}+1}-x_{{n_j}}\|+\|x_{n_j}-x_{{n_j}-1}\|),$$
consequently,
$$\lim_{j\To+\infty}\|W_{n_j}\|=0.$$

Now, one has
$$(u_{n_j},W_{n_j})\To(\ol u, 0)\mbox{ and }F(u_{n_j})\To F(\ol u),\,j\To+\infty$$
hence by  the closedness criterion of the graph of the limiting subdifferential we get
$$0\in\p F(\ol u),$$
which shows that $\ol u\in  \crit (F)$.
 \end{proof}

 Next we prove Corollary \ref{corabstrconv}.

 \begin{proof} (Proof of Corollary \ref{corabstrconv}) The claim that (H3) holds with $c_1=2+c$ and $c_2=c$ is an easy verification. We have to show that \eqref{e00} holds, that is, $u_n\in B(u^*,\rho)$ implies $u_{n+1}\in B(u^*,\sigma)$ for all $n\in\N.$

According to (H1), the assumption that $F(u_n)\ge F(u^*)$ for all $n\ge 1$ and  the hypotheses of Lemma \ref{abstrconv},  we have
$$\|x_n-x_{n-1}\|\le\sqrt{\frac{F(u_{n-1})-F(u_{n})}{a}}\le\sqrt{\frac{F(u_0)-F(u_{n})}{a}}\le\sqrt{\frac{F(u_0)-F(u^*)}{a}}<\sqrt{\frac{\eta}{a}}$$
and
$$\|x_{n+1}-x_{n}\|\le\sqrt{\frac{F(u_{n})-F(u_{n+1})}{a}}\le\sqrt{\frac{F(u_0)-F(u_{n+1})}{a}}\le\sqrt{\frac{F(u_0)-F(u^*)}{a}}<\sqrt{\frac{\eta}{a}}$$
for all $n\ge 1$.

Assume now that $n\ge 1$ and $u_n\in B(u^*,\rho).$
Then, by using the triangle inequality we get
$$\|u_{n+1}-u^*\|=\|(u_{n+1}-u_n)+(u_n-u^*)\|\le \|u_{n+1}-u_n\|+\|u_n-u^*\|\le\|u_{n+1}-u_n\|+\rho.$$

Further,
\begin{align}\nonumber\|u_{n+1}-u_n\|&=\|(v_{n+1}-v_n,w_{n+1}-w_n)\|\\
\nonumber&\le\|x_{n+1}+\a_{n+1}(x_{n+1}-x_{n})-x_n-\a_n(x_n-x_{n-1})\|\\
\nonumber&+\|x_{n+1}+\b_{n+1}(x_{n+1}-x_{n})-x_n-\b_n(x_n-x_{n-1})\|\\
\nonumber&\le(2+|\a_{n+1}|+|\b_{n+1}|)\|x_{n+1}-x_n\|+(|\a_n|+|\b_n|)\|x_n-x_{n-1}\|\\
\nonumber&\le(2+c)\|x_{n+1}-x_n\|+c\|x_{n}-x_{n-1}\|,
\end{align}
where $c=\sup_{n\in\N}(|\a_{n}|+|\b_{n}|).$

Consequently, we have
$$\|u_{n+1}-u^*\|\le (2+c)\|x_{n+1}-x_n\|+c\|x_{n}-x_{n-1}\|+\rho<(2+2c)\sqrt{\frac{\eta}{a}}+\rho\le\sigma,$$ which
is exactly $u_{n+1}\in B(u^*,\sigma).$ Further, arguing analogously as at Step I. in the proof of Lemma \ref{abstrconv}, we obtain that $u_{1}\in  B(u^*,\rho)\subseteq B(u^*,\sigma)$ and this concludes the proof.
\end{proof}
%\end{comment}

Now we are ready to prove Theorem \ref{thabstrconv}.

\begin{proof}(Theorem \ref{thabstrconv})
We will apply Corollary \ref{corabstrconv}. Since $u^*=(x^*,x^*)\in \omega((u_n)_{n\in\N})$ there exists a subsequence
$(u_{n_k})_{k\in\N}$ such that
$$u_{n_k}\To u^*,\, k\To+\infty.$$

From (H1) we get that the sequence $(F(u_n))_{n\in\N}$ is decreasing and  from (H4), which according to the hypotheses holds for $u^*$, one has $F(u_{n_k})\To F(u^*),\, k\To+\infty,$ that implies
\begin{equation}\label{tce1}
F(u_n)\To F(u^*),\,n\To+\infty\mbox{ and }F(u_n)\ge F(u^*),\,\mbox{for all }n\in\N.
\end{equation}

We show next that $x_{n_k}\To x^*,\, k\To+\infty.$ Indeed, from (H1) one has
$$a\|x_{n_k}-x_{{n_k}-1}\|^2\le F(u_{n_k-1})-F(u_{{n_k}})$$
and obviously the right side of the above inequality goes to $0$ as $k\To+\infty.$
Hence,
$$\lim_{k\To+\infty}(x_{n_k}-x_{n_k-1})=0.$$
Further, since the sequences $(\a_n)_{n\in\N},(\b_n)_{n\in\N}$ are bounded we get
$$\lim_{k\To+\infty}\a_{n_k}(x_{n_k}-x_{n_k-1})=0$$
and
$$\lim_{k\To+\infty}\b_{n_k}(x_{n_k}-x_{n_k-1})=0.$$
Finally, $u_{n_k}\To u^*,\, k\To+\infty$ is equivalent to
$$x_{n_k}-x^*+\a_{n_k}(x_{n_k}-x_{{n_k}-1})\To 0,\, k\To+\infty$$
and
$$x_{n_k}-x^*+\b_{n_k}(x_{n_k}-x_{{n_k}-1})\To 0,\, k\To+\infty,$$
which lead to the desired conclusion, that is
\begin{equation}\label{tce2}
x_{n_k}\To x^*,\, k\To+\infty.
\end{equation}
The KL property around $u^*$ states the existence of quantities $\varphi$, $U$, and $\eta$ as in Definition \ref{KL-property}.
Let $\sigma > 0 $ be such that $ B(u^*, \sigma)\subseteq U$ and $\rho\in(0,\sigma).$  If necessary we shrink $\eta$ such that
$\eta < \frac{a(\sigma-\rho)^2}{4(1+c)^2},$ where $c=\sup_{n\in\N}(|\a_n|+|\b_n|).$

Now, since the function $\varphi$ is continuous and $(F(u_n))$ is nonincreasing, further $F(u_n)\To F(u^*),\,n\To+\infty$,  $\varphi(0)=0$ and $u_{n_k}\To u^*,\,x_{n_k}\To x^*,\, k\To+\infty$ we conclude that there exists
$n_0\in\N,\,n_0\ge 1$ such that $u_{n_0}\in B(u^*,\rho)$ and $F(u^*)\le F(u_{n_0})<F(u^*)+\eta,$
moreover
$$\|x^*-x_{n_0}\|+3\sqrt{\frac{F(u_{n_0})-F(u^*)}{a}}+\frac{9b}{4a}\varphi(F(u_{n_0})-F(u^*))<\frac{\rho}{c_1+c_2}.$$

Hence, Corollary \ref{corabstrconv} and consequently Lemma \ref{abstrconv} can be applied to the sequence $(\mathcal{U}_n)_{n\in\N},\, \mathcal{U}_n=u_{n_0+n}.$

Thus, according to Lemma \ref{abstrconv}, $(\mathcal{U}_n)_{n\in\N}$ converges to a point $(\ol x,\ol x)\in\crit(F),$ consequently  $(u_n)_{n\in\N}$ converges to $(\ol x,\ol x).$ But then, since $\omega((u_n)_{n\in\N})=\{(\ol x,\ol x)\}$ one has $x^*=\ol x.$  Hence,
$(x_n)_{n\in\N}$ converges to $x^*$, $(u_n)_{n\in\N}$ converges to $u^*$ and $u^*\in\crit(F).$
\end{proof}

\section{Abstract convergence rates in terms of the KL exponent}

The following lemma was established in \cite{BN} and will be crucial in obtaining our convergence rates, (see also  \cite{attouch-bolte2009} for different techniques).

\begin{lemma}[\cite{BN} Lemma 15]\label{ratel} Let $(e_n)_{n\in\N}$ be  a  monotonically  decreasing  positive sequence  converging to $0.$  Assume  further that there exist the natural numbers $l_0\ge 1$ and $n_0\ge l_0$  such that for every $n\ge n_0$ one has
\begin{equation}\label{eseq}
e_{n-l_0}-e_n\ge C_0 e_n^{2\t}
\end{equation}
where $C_0>0$ is some constant and $\t\in[0,1).$  Then following  statements are true:
\begin{itemize}
\item[(i)]  if $\t=0,$  then $(e_n)_{n\ge n_0}$ converges in finite time;
\item[(ii)]  if $\t\in\left(0,\frac12\right]$, then there exists $C_1>0$ and $Q\in[0,1)$, such that for every $n\ge n_0$
$$e_n\le C_1 Q^n;$$
\item[(iii)]  if $\t\in\left[\frac12,1\right)$, then there exists $C_2>0$, such that for every $n\ge n_0+l_0$
$$e_n\le C_2(n-l_0+1)^{-\frac{1}{2\t-1}}.$$
\end{itemize}
\end{lemma}

Now we are ready  to prove Theorem \ref{thabstrrate}.

\begin{proof}(\rm Proof of Theorem \ref{thabstrrate}.)

The fact that the sequence $(x_n)_{n\in\N}$ converges to $x^*$, $(u_n)_{n\in\N}$ converges to $u^*$ and $u^*\in\crit(F)$ follows from Theorem \ref{thabstrconv}.
We divide the proof of the statements (a)-(c) into two cases.

{\bf Case I.} Assume that there exists $\ol n\in\N$ such that $F(u_{\ol n})=F(u^*).$

According to (H4) there exists $(u_{n_j})\subseteq (u_n)$ such that
$$u_{n_j}\To u^*,\,F(u_{n_j})\To F(u^*),\,j\To+\infty.$$
Now, $F(u_{n_j})=F(u^*)$ for all $n_j\ge \ol n$ since the sequence $(F(u_{n_j}))_{j\in\N}$ is decreasing, and hence
$$F(u^*)=F(u_{\ol n})\ge F(u_{n_j})\ge\lim_{j\To+\infty}F(u_{n_j})=F(u^*).$$
Further, for every $n\ge \ol n$ there exists $j_0\in\N$ such that $n\le n_{j_0}$, consequently
$$F(u^*)=F(u_{\ol n})\ge F(n)\ge F(u_{n_{j_0}})=F(u^*).$$
In other words, $F(u_n)=F(u^*)$ for all $n\ge \ol n.$ From (H1) we get that for all $n\ge \ol n$
$$\|x_{n+1}-x_n\|^2\le\frac{1}{a}(F(u_n)-F(u_{n+1})=F(u^*)-F(u^*)=0$$
hence, $x_{n+1}=x_n$ for all $n\ge \ol n.$ But $x_n\To x^*,\,n\To+\infty$, hence $x_n=x^*$ for all $n\ge \ol n$.

But then, $u_n=(x_n,x_n)=(x^*,x^*)$ for all $n\ge \ol n$.
Consequently, $(F(u_n))_{n\in\N},(x_n)_{n\in\N}$ and $(u_n)_{n\in\N}$ converge in a finite number of steps and this concludes $(a)-(c)$.

{\bf Case II.} We assume that $F(u_n)>F(u^*)$ for all $n\in\N.$ Now, by using (H2) and (H1) we get
\begin{align}\label{foren}
\dist\nolimits^2(0,\p F(u_n))&\le b^2(\|x_n-x_{n-1}\|+\|x_{n-1}-x_{n-2}\|)^2\\
\nonumber&\le 2b^2(\|x_n-x_{n-1}\|^2+\|x_{n-1}-x_{n-2}\|^2)\\
\nonumber&\le \frac{2b^2}{a}((F(u_{n-1})-F(u_n))+(F(u_{n-2})-F(u_{n-1})))\\
\nonumber&=\frac{2b^2}{a}((F(u_{n-2})-F(u^*))-(F(u_n)-F(u^*))),
\end{align}
for all $n\ge 2.$

Now, according to (H4) there exists $(u_{n_j})\subseteq (u_n)$ such that
$$u_{n_j}\To u^*,\,F(u_{n_j})\To F(u^*),\,j\To+\infty.$$
Combining the above fact with the facts that $(F(u_n))$ is nonincreasing and $u_n\To u^*,\,n\To+\infty$ we conclude that there exists $\ol {n}\in\N,\,\ol n\ge 2$ such that $F(u^*)<F(u_n)<F(u^*)+\eta$ and $u_n\in B(u^*,\e)$ for all $n\ge \ol n.$ So,  since the function $F$ has the Kurdyka-{\L}ojasiewicz property with an exponent $\t\in[0,1)$ at $u^*$ we can apply the KL-inequality and we get
\begin{equation}\label{KLforF}
\dist\nolimits^2(0, \p F(u_n)) \ge \frac{1}{K^2}(F(u_n)-F(u^*))^{2\t},\mbox{ for all }n\ge\ol n.
\end{equation}

Hence, \eqref{foren} and \eqref{KLforF} yields
\begin{equation}\label{foren1}
\frac{a}{2b^2K^2}(F(u_n)-F(u^*))^{2\t}\le (F(u_{n-2})-F(u^*))-(F(u_n)-F(u^*)),\mbox{ for all }n\ge\ol n.
\end{equation}

Further, using (H4) again, we have $F(u_{n_j})\To F(u^*),\,j\To+\infty$ and $(F(u_n))$ is nonincreasing  which leads to
$$\lim_{n\To+\infty}F(u_n)-F(u^*)=0.$$

Let us denote $e_n=F(u_n)-F(u^*).$ Then  $(e_n)_{n\in\N}$ is a monotonically  decreasing  positive sequence  converging to $0.$  Further from \eqref{foren1} we have  that there exist the natural numbers $l_0=2$ and $\ol n\ge l_0$  such that for every $n\ge \ol n$ one has
$$e_{n-l_0}-e_n\ge C_0 e_n^{2\t},$$
where $C_0=\frac{a}{2b^2K^2}>0.$
Consequently, Lemma \ref{ratel} can be applied.
\vskip0.3cm
Let $\t=0$. Then, the sequence $(F(u_n)-F(u^*))$ converges in a finite number of steps, that is $F(u_n)=F(u^*)$ after and index $n_1\in\N$. But then, according to Case I. $(x_n)$  and $(u_n)$ converges in a finite number of steps and this concludes (a).
\vskip0.3cm
Let $\t\in\left(0,\frac12\right].$ Then, there exists $C_1>0$ and $Q\in[0,1)$, such that for every $n\ge \ol n$
$$F(u_n)-F(u^*)\le C_1 Q^n.$$

According to \eqref{fontos} we have
$$3\|x_{k+1}-x_{k}\|\le\|x_{k}-x_{k-1}\|+\|x_{k-1}-x_{k-2}\|+\frac{9bK}{4a(1-\t)}(e_k^{1-\t}-e_{k+1}^{1-\t})$$
for all $k\ge \ol n.$ Summing up the latter relation from $k=n\ge\ol n$ to $k=P>n$ we get
$$\sum_{k=1}^P\|x_{k+1}-x_{k}\|\le2\|x_{n}-x_{n-1}\|+\|x_{n-1}-x_{n-2}\|-2\|x_{P+1}-x_{P}\|-\|x_{P}-x_{P-1}\|+
\frac{9bK}{4a(1-\t)}(e_n^{1-\t}-e_{P+1}^{1-\t}).$$
Now, from the triangle inequality we have
$$\|x_n-x_{P+1}\|\le \sum_{k=1}^P\|x_{k+1}-x_{k}\|,$$
hence,
$$\|x_n-x_{P+1}\|\le2\|x_{n}-x_{n-1}\|+\|x_{n-1}-x_{n-2}\|-2\|x_{P+1}-x_{P}\|-\|x_{P}-x_{P-1}\|+
\frac{9bK}{4a(1-\t)}(e_n^{1-\t}-e_{P+1}^{1-\t}).$$
By neglecting the nonpositive terms and letting $P\To+\infty$ we get
\begin{equation}\label{fort12}
\|x_n-x^*\|\le 2\|x_{n}-x_{n-1}\|+\|x_{n-1}-x_{n-2}\|+\frac{9bK}{4a(1-\t)}e_n^{1-\t}.
\end{equation}
Now, by using (H1) we get
\begin{equation}\label{fort120}
2\|x_{n}-x_{n-1}\|+\|x_{n-1}-x_{n-2}\|\le\sqrt{\frac{2(e_{n-1}-e_{n})}{a}}+\sqrt{\frac{e_{n-2}-e_{n-1}}{a}}.
\end{equation}
Consequently, for all $n\ge \ol n+2$ one has
\begin{equation}\label{fort121}
\|x_n-x^*\|\le \sqrt{\frac{2}{a}}\sqrt{C_1}Q^{\frac{n-1}{2}}+ \sqrt{\frac{1}{a}}\sqrt{C_1}Q^{\frac{n-2}{2}}+\frac{9bK}{4a(1-\t)}C_1^{1-\t}Q^{(1-\t)n}.
\end{equation}
Now, $\t\le\frac12$ and $Q\in[0,1)$ hence, $Q^{(1-\t)n}\le Q^{\frac{n}{2}},$ hence \eqref{fort121} yields
$$\|x_n-x^*\|\le \ol C Q^{\frac{n}{2}}$$
for some $\ol C$ and for all $n\ge \ol n+2.$

Further, according to (H3)
\begin{align}\label{foru}
\|u_n-u^*\|&\le(2+c)\|x_n-x^*\|+c\|x_{n-1}-x^*\|\\
\nonumber&\le (2+c)C Q^{\frac{n}{2}}+cC Q^{\frac{n-1}{2}}\\
\nonumber&= \ol C_1 Q^{\frac{n}{2}},
\end{align}
for all $n\ge \ol n+3$, where $c=\sup_{n\in\N}(|\a_n|+|\b_n|).$ Hence, (b) is complete if one takes $A_1=\max(C_1,\ol C,\ol C_1)$ and $\ol k=\ol n+3$.
\vskip0.3cm

Let $\t\in\left[\frac12,1\right)$. Then, there exists $C_2>0$, such that for every $n\ge \ol n+2$
$$F(u_n)-F(u^*)\le C_2(n-1)^{-\frac{1}{2\t-1}}.$$
But $$(n-1)^{-\frac{1}{2\t-1}}\le 2^{\frac{1}{2\t-1}}n^{-\frac{1}{2\t-1}},$$
hence
$$F(u_n)-F(u^*)\le C_2 2^{\frac{1}{2\t-1}}n^{-\frac{1}{2\t-1}}=\ol C_2 n^{-\frac{1}{2\t-1}},$$
for all $n\ge \ol n+2.$

Now, by using \eqref{fort12} and \eqref{fort120} we get
\begin{equation}\label{fort1}
\|x_n-x^*\|\le \sqrt{\frac{2(F(u_{n-1})-F(u^*))}{a}}+\sqrt{\frac{F(u_{n-2})-F(u^*)}{a}}+\frac{9bK}{4a(1-\t)}(F(u_{n-1})-F(u^*))^{1-\t}.
\end{equation}

Hence, there exists $\ol C_3,\ol C_4,\ol C_5>0$ such that
$$\|x_n-x^*\|\le \ol C_3 n^{-\frac{\frac12}{2\t-1}}+\ol C_4 n^{-\frac{\frac12}{2\t-1}}+\ol C_5 n^{-\frac{1-\t}{2\t-1}},$$
or all $n\ge \ol n+4.$
Now, since $\t>\frac12$ one has
$$n^{-\frac{1-\t}{2\t-1}}\ge n^{-\frac{\frac12}{2\t-1}}$$
hence there exists $\ol A_2>0$ such that
$$\|x_n-x^*\|\le \ol A_2 n^{-\frac{1-\t}{2\t-1}},$$
or all $n\ge \ol n+4.$

By using the form of $u_n$
$$\|u_n-u^*\|\le(2+c)\|x_n-x^*\|+c\|x_{n-1}-x^*\|$$
or all $n\in\N$, where $c=\sup_{n\in\N}(|\a_n|+|\b_n|),$ hence there exists $\ol A_3$ such that
$$\|u_n-u^*\|\le  \ol A_3 n^{-\frac{1-\t}{2\t-1}},$$
or all $n\ge \ol n+5.$

Consequently, (c) holds for $A_2=\max(\ol C_2,\ol A_2,\ol A_3)$ and $\ol k=\ol n+5.$
\end{proof}

\end{document}